\documentclass{amsart}
\usepackage{amssymb}
\usepackage[matrix,arrow,curve]{xy}
\usepackage{verbatim}
\usepackage{graphics}
\usepackage{a4wide}
\usepackage[utf8]{inputenc}
\newtheorem{theorem}{Theorem}[section]
\newtheorem{lemma}[theorem]{Lemma}
\newtheorem{remark}[theorem]{Remark}

\newtheorem{corollary}[theorem]{Corollary}

\newtheorem{proposition}[theorem]{Proposition}

\usepackage{tikz}

\def\dd{\mathrm{d}}

\def\eps{\varepsilon}

\def\Z{\mathbf{Z}}
\def\N{\mathbf{N}}

\def\R{\mathbf{R}}

\def\Hom{\mathrm{Hom}}

\def\into{\hookrightarrow}

\def\la{\lambda}
\def\supp{\mathrm{supp}}

\def\KK{\mathcal{K}}
\def\LL{\bar{L}}
\def\PF{\mathcal{P}_{\mathrm{f}}}

\date{September 25, 2020.}

\begin{document}
\centerline{}

\title[Simplicial Random Variables]{Simplicial Random Variables}

\author[I.~Marin]{Ivan Marin}
\address{LAMFA, UMR CNRS 7352, Universit\'e de Picardie-Jules Verne, Amiens, France}
\email{ivan.marin@u-picardie.fr}
\medskip

\begin{abstract} 
We introduce a new `geometric realization' of an (abstract) simplicial complex, inspired by probability
theory. This space (and its completion) is a metric space, which has the right (weak) homotopy type,
and which can be compared with the usual geometric realization through a natural map, which has probabilistic meaning : it associates to a random variable its probability mass function. This `probability law' map
is proved to be a (Serre) fibration and a (weak) homotopy equivalence. 
\end{abstract}

\maketitle

\tableofcontents

\section{Introduction and main results}

In this paper we consider a new `geometric realization' of an (abstract) simplicial complex, inspired by probability
theory. This space is a metric space, which has the right (weak) homotopy type,
and can be compared with the usual geometric realization through a map, which is very natural
in probabilistic terms : it associates to a random variable its probability mass function. This `probability law' function
is proved to be a (Serre) fibration and a (weak) homotopy equivalence.  This construction passes to the completion,
and has nice functorial properties.

We specify the details now. Let $S$ be a set, and $\PF(S)$ the set of its finite subsets. We set $\PF^*(S) = \PF(S) \setminus \{ \emptyset \}$. Recall that an (abstract) simplicial complex is a collection of subsets $\mathcal{K} \subset \PF^*(S) $ with the
property that, for all $X \in \mathcal{K}$ and $Y \in \PF^*(S)$, $Y \subset X \Rightarrow Y \in \mathcal{K}$. The elements
of $\mathcal{K}$ are called its faces, and the vertices of $\mathcal{K}$ are the union of the
elements of $\mathcal{K}$.

We endow $S$ with the discrete metric of diameter $1$, and with the Borel $\sigma$-algebra associated to this
topology. We let $\Omega$ denote a nonatomic standard  probability space with measure $\la$. Recall that all such probability spaces are isomorphic and can be identified in particular with any hypercube $[0,1]^n$, $n \geq 1$ endowed with the Lebesgue measure. We define $L(\Omega,S)$ as
the set of random variables $\Omega \to S$, that is the set of measurable maps $\Omega \to S$ modulo the equivalence relation $f \equiv g$ if $f$ and $g$ agree almost everywhere, that is $\la(\{ x; f(x) = g(x) \}) = 0$. We consider it as a metric space,
endowed with the metric 
$$
d(f,g) = \int_{\Omega}d(f(t),g(t)) \dd t
= \la\left(\{ x\in \Omega ; f(x) \neq g(x) \}\right).
$$

We define $L(\Omega,\mathcal{K})$ as the subset of $L(\Omega,S)$ made of the (equivalence
classes of) measurable maps
$f : \Omega \to S$ such that $\{ s \in S \ | \ \la(f^{-1}(\{ s \})) > 0 \} \in \mathcal{K}$.

Recall that the (usual) `geometric' realization of $\mathcal{K}$ is defined as
$$
|\mathcal{K}| = \{ t : S \to [0,1] \ | \ \{ s \in S ; t_s > 0 \} \in \mathcal{K} \ \& \ \sum_{s \in S} t_s = 1 \}
$$
and that its topology is given by the direct limit of the $[0,1]^A$ for $A \in \PF(S)$.
There is a natural map $L(\Omega,\mathcal{K}) \to |\mathcal{K}|$
which associates to $f : \Omega \to \mathcal{K}$
the element $t : S \to [0,1]$ defined by $t_s = \la(f^{-1}(\{ s \}))$. In probabilistic
terms, it associates to the random variable $f$ its probability law, or probability mass function. We denote $|\KK|_1$
the same set as $|\KK|$, but with the topology defined by the metric $|\alpha- \beta|_1 = \sum_{s \in S} |\alpha(s) - \beta(s)|$.
We denote $\overline{|\KK|_1}$ its completion as a metric space.

It is easily checked that, unless $S$ is finite, $L(\Omega,\mathcal{K})$
is not in general closed in $L(\Omega,S)$, and therefore not complete.
We denote $\LL(\Omega,\KK)$ its closure inside $L(\Omega,S)$. The `probability law' map $\Psi : L(\Omega,\KK) \to |\KK|_1$
is actually continuous, and can be extended to a map $\overline{\Psi} : \LL(\Omega,\KK) \to \overline{|\KK|}_1$.
Keane's Theorem about the contractibility of $\mathrm{Aut}(\Omega)$ (see \cite{KEANE}) easily implies that these maps have contractible fibers. The goal of this note is
to specify the homotopy-theoretic features of them. We get the following results.

\begin{theorem} \label{theo:maintheorem} {\ }
\begin{enumerate}
\item The map $L(\Omega,\mathcal{K}) \to \overline{L}(\Omega,\mathcal{K})$
is a weak homotopy equivalence.
\item The `probability law' map $L(\Omega,\KK) \to |\KK|_1$ is a Serre fibration and a weak homotopy
equivalence. It admits a continuous global section.
\item The `probability law' map $\LL(\Omega,\KK) \to \overline{|\KK|_1}$ is a Serre fibration and a weak homotopy
equivalence. It admits a continuous global section.
\item $L(\Omega,\KK)$ and $\LL(\Omega,\KK)$ have the same weak homotopy type as the `geometric realization'
$|\KK|$ of $\KK$.
\end{enumerate}
\end{theorem}

In other terms, in the commutative diagram below, the vertical maps are Serre fibrations, and all the maps involved
are weak homotopy equivalences, the map $|\KK| \to |\KK|_1$ being in addition a strong homotopy equivalence.
Actually, when $\KK$ is finite, we prove that all the maps are strong homotopy equivalences (see Theorem \ref{theo:fibrations}).

$$
\xymatrix{
 & L(\Omega,\KK) \ar@{^(->}[r] \ar[d]_{\Psi_{\KK}}  & \LL(\Omega,\KK) \ar[d]_{\bar{\Psi}_{\KK}} \\
|\KK| \ar[r]_{}  & |\KK|_1 \ar@{^(->}[r] & \overline{|\KK|_1} 
}
$$

We now comment on the functorial properties of this construction. By definition, a morphism
$\varphi : \mathcal{K}_1 \to \mathcal{K}_2$ between simplicial complexes is a map
from the set $\bigcup \mathcal{K}_1$ of vertices of $\KK_1$ to the set of vertices
of $\KK_2$ with the property that $\forall F \in \mathcal{K}_1 \ \ \varphi(F) \in \KK_2$. We denote
$\mathbf{Simp}$ the corresponding category of simplicial complexes. For such an abstract simplicial complex $\mathcal{K}$, our space $L(\Omega,\mathcal{K})$ as for ambient space $L(\Omega,S)$ with $S = \bigcup \mathcal{K}$ the set of vertices of $\mathcal{K}$.

Let $\mathbf{Set}$ denote the category of sets and
$\mathbf{Met}_1$ denote the full subcategory of the category of metric spaces and contracting maps
made of the spaces of diameter at most $1$. Here a map $f : X \to Y$ between two metric spaces is called contracting
if $\forall a,b \in X \ d(f(a),f(b)) \leq d(a,b)$. 
 Let $\mathbf{CMet}_1$ be the full subcategory of $\mathbf{Met}_1$
made of complete metric spaces. There is a completion functor $Comp : \mathbf{Met}_1 \to \mathbf{CMet}_1$
which associates to each metric space its completion.
Then $L(\Omega,\bullet) : X \leadsto L(\Omega,X)$
defines a functor $\mathbf{Set} \to \mathbf{CMet}_1$ (see \cite{CCS}). It can be decomposed as $L(\Omega,\bullet) = Comp \circ L_{\mathrm{f}}(\Omega,
\bullet)$ where $L_{\mathrm{f}}(\Omega,S)$ is the subspace of $L(\Omega,S)$ made of the (equivalence classes of) functions $f : \Omega \to S$ of essentially finite image, that is such that there exists $S_0 \subset S$ finite such that $\sum_{s \in S_0} \la(f^{-1}(\{ s \})) = 1$.

We prove in section \ref{sect:functorial} below that our simplicial constructions have
similar functorial properties, which can be summed up as follows.

\begin{proposition} \label{prop:functorial}
 $L(\Omega,\bullet)$ and $\LL(\Omega,\bullet)$ define functors
$\mathbf{Simp} \to \mathbf{Met}_1$ and $\mathbf{Simp} \to \mathbf{CMet}_1$, with
the property that $\LL(\Omega,\bullet) = Comp \circ L(\Omega,\bullet)$.
\end{proposition}

\bigskip

{\bf Acknowledgements.} I thank D. Chataur and A. Rivière for discussions and references. I foremost thank an anonymous referee for a careful check.

\section{Simplicial properties and completion}

In this section we prove part (1) of Theorem \ref{theo:maintheorem}.
We start by proving the
functorial properties stated in the introduction.

\subsection{Functorial properties }
We denote, as in the previous section,
$\bar{L}(\Omega,\mathcal{K})$ the closure of
$L(\Omega,\mathcal{K})$ inside $L(\Omega,S)$. As a closed subset of a complete metric space,
it is a complete metric space.
For any $f \in L(\Omega,S)$, we denote
$$
f(\Omega) = \{ s \in S \ | \ \la(f^{-1}(\{ s \} ))> 0 \}
$$ 
the essential image of an arbitrary measurable map $\Omega \to S$ representing $f$.

\begin{lemma} \label{lem:caractfLbar}  Let $f \in L(\Omega,S)$. Then $f \in \overline{L}(\Omega,\mathcal{K})$
if and only if every nonempty finite subset of $f(\Omega)$
belongs to $\mathcal{K}$.
\end{lemma}
\begin{proof}
Assume $f \in \bar{L}(\Omega,\mathcal{K})$ and let $F \subset f(\Omega)$ be a nonempty finite subset
as in the statement. We set $m = \min \{ \la(f^{-1}(\{ s \})) \ | \ s \in F \}$.
We have $m > 0$. Since $f \in \bar{L}(\Omega,\mathcal{K})$, there exists $f_0 \in L(\Omega,\mathcal{K})$ such that $d(f,f_0) < m$. 
We then have $F \subset f_0(\Omega)$.
Indeed, there would otherwise exist $s \in F \setminus f_0(\Omega)$, and then
$d(f,f_0) \geq \la(f^{-1}(\{ s \})) \geq m$, a contradiction. From this we get $F \in \mathcal{K}$.
Conversely, assume 
that 
every nonempty finite subset of $f(\Omega)$
belongs to $\mathcal{K}$.
From \cite{CCS} Proposition 3.3 we know that $ f(\Omega) \subset S$
is countable. If $f(\Omega)$ is finite we have $f(\Omega) \in \mathcal{K}$ by assumption and $f \in L(\Omega,\mathcal{K})$.
Otherwise, let us fix a bijection $\N \to f(\Omega)$, $n \mapsto x_n$
and define $f_n \in L(\Omega,S)$ by $f_n(t) = f(t)$ if $f(t) \in \{x_0,\dots,x_n \}$,
and $f_n(t) = x_0$ otherwise. Clearly $f_n(\Omega) \subset f(\Omega)$
is nonempty finite hence belongs to $\mathcal{K}$, and
$f_n \in L(\Omega,\mathcal{K})$.  
On the other hand,  $d(f_n,f) \leq \sum_{k > n } \la(f^{-1}(\{ x_k \})) \to 0$, hence
$f \in \bar{L}(\Omega,\mathcal{K})$ and this proves the claim.
\end{proof}

\label{sect:functorial} 

We prove that, as announced in the introduction, $\bar{L}(\Omega,\bullet)$
provides a functor $\mathbf{Simp} \to \mathbf{CMet}_1$ that can be decomposed as
$Comp \circ L(\Omega,\bullet)$, where $L(\Omega,\bullet)$ is itself a functor $\mathbf{Simp} \to \mathbf{Met}_1$.

Let $\varphi \in \Hom_{\mathbf{Simp}}(\mathcal{K}_1,\mathcal{K}_2)$ that is $\varphi : \bigcup \mathcal{K}_1 \to \bigcup \mathcal{K}_2$ such that $\varphi(F) \in \mathcal{K}_2$ for all $F \in \mathcal{K}_1$.  If $f \in L(\Omega,\KK_1)$, $g = L(\Omega,\varphi)(f) = \varphi \circ f$
is a measurable map and $g(\Omega) = \varphi(f(\Omega))$.
Since $f(\Omega) \in \mathcal{K}_1$
and $\varphi$ is simplicial we get that $\varphi(f(\Omega)) \in \KK_2$
hence $g \in L(\Omega,\KK_2)$. From this one gets immediately that
$L(\Omega,\bullet)$ indeed defines a functor $\mathbf{Simp} \to \mathbf{Met_1}$.

Similarly, if $f \in \bar{L}(\Omega,\KK_1)$ and $g = \varphi \circ f = L(\Omega,\varphi)(f) \in L(\Omega,S)$,
then again $g(\Omega) 
= \varphi(f(\Omega))$. But, for any finite set
$F \subset g(\Omega) = \varphi(f(\Omega))$ there exists $F' \subset f(\Omega)$ finite
and with the property that $F = \varphi(F')$. Now $f \in \LL(\Omega,\KK_1) \Rightarrow F' \in \KK_1$,
by Lemma \ref{lem:caractfLbar}, hence $F \in \KK_2$ because $\varphi$ is a simplicial morphism. By Lemma \ref{lem:caractfLbar} one
gets $g \in \LL(\Omega,\KK_2)$, hence $\LL(\Omega,\bullet)$
defines a functor $\mathbf{Simp} \to \mathbf{CMet_1}$. We checks immediately that $\LL(\Omega,\bullet) = Comp \circ L(\Omega,\bullet)$, and this proves Proposition \ref{prop:functorial}.

 \subsection{Technical preliminaries}

We denote by $2$ in the notation $L(\Omega,2)$ a set
with two elements. When needed, we will also assume that this set is pointed, that is contains a special point called $0$, so that $f \in L(\Omega,2)$ can be identified with $\{ t \in \Omega ; f(t) \neq 0 \}$, up
to a set of measure $0$. Note that these conventions agree with the set-theoretic definition of $2 = \{ 0,1 \} = \{ \emptyset,\{ \emptyset \} \}$. 

\begin{lemma} \label{lem:omegaF} Let $F$ be a set. The map $f \mapsto \{ t \in \Omega ; f(t) \not\in F \}$
is uniformly continuous $L(\Omega,S) \to L(\Omega,2)$, and even contracting.
\end{lemma}
\begin{proof}
Let $f_1,f_2 \in L(\Omega, S)$, and $\Psi : L(\Omega,S) \to L(\Omega,2)$
the map defined by the statement. Then $\Psi(f_1)(t) \neq \Psi(f_2)(t) \Rightarrow f_1(t) \neq f_2(t)$,
hence $d(\Psi(f_1)(t),\Psi(f_2)(t)) \leq d(f_1(t),f_2(t))$ for all $t \in \Omega$
and finally $d(\Psi(f_1),\Psi(f_2)) \leq d(f_1,f_2)$, whence $\Psi$
is contracting and uniformly continuous.
\end{proof}

\begin{lemma} \label{lem:ineq4} Let $a,b,c,d \in \R$ with $a \leq b$ and $c \leq d$.
Then
$$
\la\left( [a,b] \setminus ]c,d[
\right)
 \leqslant |a-c| + |b-d|.$$

\end{lemma}
\begin{proof}
There are six possible relative positions of $c \leq d$ with respect to $a \leq b$ to consider, which are depicted as follows.
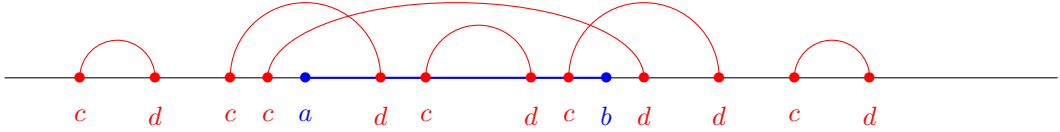
\begin{center}
\begin{tikzpicture}
\draw[dashed] (-4,0) -- (10,0);
\draw[ultra thick,blue] (0,0) -- (4,0);
\draw[thick,blue] (0,0) node {$\bullet$};
\draw[thick,blue] (0,-0.5) node {$a$};
\draw[thick,blue] (4,0) node {$\bullet$};
\draw[thick,blue] (4,-0.5) node {$b$};
\draw[thick,red] (-3,0) node {$\bullet$};
\draw[thick,red] (-3,-0.5) node {$c$};
\draw[thick,red] (-2,0) node {$\bullet$};
\draw[thick,red] (-2,-0.5) node {$d$};
\draw[thick,red] (-1,0) node {$\bullet$};
\draw[thick,red] (-1,-0.5) node {$c$};
\draw[thick,red] (1,0) node {$\bullet$};
\draw[thick,red] (1,-0.5) node {$d$};
\draw[thick,red] (-0.5,0) node {$\bullet$};
\draw[thick,red] (-0.5,-0.5) node {$c$};
\draw[thick,red] (4.5,0) node {$\bullet$};
\draw[thick,red] (4.5,-0.5) node {$d$};
\draw[thick,red] (1.6,0) node {$\bullet$};
\draw[thick,red] (1.6,-0.5) node {$c$};
\draw[thick,red] (3,0) node {$\bullet$};
\draw[thick,red] (3,-0.5) node {$d$};
\draw[thick,red] (3.5,0) node {$\bullet$};
\draw[thick,red] (3.5,-0.5) node {$c$};
\draw[thick,red] (5.5,0) node {$\bullet$};
\draw[thick,red] (5.5,-0.5) node {$d$};
\draw[thick,red] (6.5,0) node {$\bullet$};
\draw[thick,red] (6.5,-0.5) node {$c$};
\draw[thick,red] (7.5,0) node {$\bullet$};
\draw[thick,red] (7.5,-0.5) node {$d$};

\begin{scope}
\clip (-4,0) -- (10,0) --  (10,4) -- (-4,4) -- cycle;
\draw[red] (-2.5,0) ellipse (.5 and .5);
\draw[red] (0,0) ellipse (1 and 1);
\draw[red] (2,0) ellipse (2.5 and 1);
\draw[red] (2.3,0) ellipse (.7 and .7);
\draw[red] (4.5,0) ellipse (1 and 1);
\draw[red] (7,0) ellipse (.5 and .5);
\end{scope}
\end{tikzpicture}
\end{center}
In three of them, namely $a \leq b \leq c\leq d$,
$c \leq d \leq a\leq b$, and $c \leq a \leq b\leq d$,
we have $\la\left( [a,b] \setminus ]c,d[
\right) = 0$. In case $c \leq a \leq d \leq b$, we have
$\la\left( [a,b] \setminus ]c,d[
\right) = \la([d,b]) = |b-d| \leq |a-c| + |b-d|$.
In case $a \leq c \leq b \leq d$, we have
$\la\left( [a,b] \setminus ]c,d[
\right) = \la([a,c]) = |a-c| \leq |a-c| + |b-d|$.
Finally, when $a \leq c \leq d \leq b$, we have
$\la\left( [a,b] \setminus ]c,d[
\right) = \la([a,c] \sqcup [d,b]) =  |a-c| + |b-d|$,
and this proves the claim.
\end{proof}

\begin{lemma} \label{lem:2rLipalphatof} Let $\Delta^r = \{ \underline{\alpha} = (\alpha_1,\dots,\alpha_r) \in \R_+^r \ | \ \alpha_1+\dots+\alpha_r = 1 \}$ denote the $r$-dimensional simplex. The map $\Delta^r \to L(\Omega,\{ 1,\dots,r \})$ defined by $\underline{\alpha} \mapsto f_{\underline{\alpha}}$
where $f_{\underline{\alpha}}(t) = i$ iff $t \in [\alpha_1+\dots+\alpha_{i-1}, \alpha_1+\dots+\alpha_{i}[$ is continuous. More precisely it is $2r$-Lipschitz if $\Delta^r$ is equipped with the metric
 $d(\underline{\alpha},\underline{\alpha'}) = \sum_i |\alpha_i - \alpha'_i|$.
\end{lemma}
\begin{proof}
We fix an identification $\Omega \simeq [0,1]$.
Let $\underline{\alpha},\underline{\alpha'} \in \Delta^r$.
We denote $\beta_i = \alpha_1+\dots+\alpha_{i}$, $\beta_0 = 0$,
and we similarly define the $\beta'_i$. We have $\beta_{i} - \beta_{i-1} = \alpha_i$
hence $|\beta'_i - \beta_i | \leq \sum_{k \leq i} |\alpha'_k - \alpha_k|$
and finally $\sum_i |\beta'_i - \beta_i | \leq r \sum_i |\alpha'_i - \alpha_i|$.
Now, for $t \in [\beta_i,\beta_{i+1}[$
we have $f_{\underline{\alpha}}(t) = f_{\underline{\alpha}'}(t)$
unless $t \not\in [\beta'_i,\beta'_{i+1}[$.
From this and Lemma \ref{lem:ineq4} we get
$$
d(f_{\underline{\alpha}},f_{\underline{\alpha'}})
\leqslant \sum_{i=1}^r \la\left([\beta_i,\beta_{i+1}[ \setminus [\beta'_i,\beta'_{i+1}[\right)
\leqslant \sum_{i=1}^r |\beta_i - \beta'_i| + |\beta_{i+1} - \beta'_{i+1}|
\leqslant 2 \sum_{i=1}^r |\beta_i- \beta'_i| \leqslant 2r \sum_{i=1}^r |\alpha_i -\alpha'_i|
$$
and this proves the claim.
\end{proof}

\begin{lemma} \label{lem:subhomotopy} Let $\mathcal{K}$ be a simplicial complex and $X$ a
topological space, and $A \subset X$. If $\gamma_0,\gamma_1 : X \to \bar{L}(\Omega,\mathcal{K})$
are two continuous maps such that $\forall x \in X \ \ \gamma_0(x)(\Omega)
\subset \gamma_1(x) (\Omega)$, and $(\gamma_0)_{|A} = (\gamma_1)_{|A}$, then $\gamma_0$ and $\gamma_1$ are homotopic
relative to $A$. Moreover, if $\gamma_0$
and $\gamma_1$ take value inside $L(\Omega,\mathcal{K})$, then the homotopy takes
values inside $L(\Omega,\mathcal{K})$.
\end{lemma}
\begin{proof}
We fix an identification $\Omega \simeq [0,1]$.
We define $H : [0,1] \times X \to L(\Omega,S)$ by
$H(u,x)(t) = \gamma_0(x)(t)$ if $t \geq u$
and $H(u,x)(t) = \gamma_1(x)(t)$ if $t < u$. We have $H(0,\bullet) = \gamma_0$
and $H(1,\bullet) = \gamma_1$.

We first check that $H$ is indeed a (set-theoretic) map $[0,1]\times X \to \bar{L}(\Omega,\mathcal{K})$.
For all $u \in [0,1]$
and $x \in X$ we have $H(u,x)(\Omega) \subset \gamma_0(x)(\Omega) \cup 
\gamma_1(x)(\Omega) = \gamma_1(x)(\Omega)$. Therefore
$H(u,x)(\Omega) \in \mathcal{K}$ if $\gamma_1(x) \in L(\Omega,\mathcal{K})$,
and all nonempty finite subsets of $H(u,x)(\Omega)\subset \gamma_1(x)(\Omega)$ belong 
to $\mathcal{K}$
if $\gamma_1(x) \in \bar{L}(\Omega,\mathcal{K})$. From this, by Lemma \ref{lem:caractfLbar} we get that
$H$ takes values inside $\bar{L}(\Omega,\mathcal{K})$, and even inside $L(\Omega,
\mathcal{K})$ if $\gamma_1 : X \to L(\Omega,\mathcal{K})$.

Now, we check that $H$ is continuous over $[0,1]\times X$. We have
$d(H(u,x),H(v,x)) \leq |u-v|$ for all $u,v \in [0,1]$ and, for all $x,y \in X$ and $u \in [0,1]$, we
have
$$
d(H(u,x),H(u,y))=\int_0^u d(\gamma_1(x)(t),\gamma_1(y)(t)) \dd t + \int_u^1 d(\gamma_0(x)(t),\gamma_0(y)(t)) \dd t
$$
{}
$$
\leqslant
\int_0^1 d(\gamma_1(x)(t),\gamma_1(y)(t)) \dd t + \int_0^1 d(\gamma_0(x)(t),\gamma_0(y)(t)) \dd t
= d(\gamma_1(x),\gamma_1(y)) + d(\gamma_0(x),\gamma_0(y))$$
from which we get $d(H(u,x),H(v,y))\leq |u-v| +  
d(\gamma_1(x),\gamma_1(y)) + d(\gamma_0(x),\gamma_0(y))$ for all $x,y \in X$ and $u,v\in [0,1]$. For any given $(u,x) \in [0,1]\times X$
this proves that $H$ is continuous at $(u,x)$. Indeed, given $\eps>0$, from the continuity of $\gamma_0,\gamma_1$ we get that, for some open neighborhood $V$ of $x$ we have $d(\gamma_0(x),
\gamma_0(y)) \leq \eps/3$ and
$d(\gamma_1(x),
\gamma_1(y)) \leq \eps/3$ for all $y \in V$. This
proves that $d(H(u,x),H(v,y))\leq \eps$ for all
$(v,y) \in ]u-\eps/3,u+\eps/3[\times V$
and this proves the continuity of $H$.

Finally, it is clear that
$\gamma_0(x) = \gamma_1(x)$ implies
$H(u,x) = \gamma_0(x) = \gamma_1(x)$ for all $u \in [0,1]$, therefore the homotopy
indeed fixes $A$.
\end{proof}

\subsection{Weak homotopy equivalence}

We now prove part (1) of the main theorem, through a series of propositions, which might be of independent interest.
\begin{proposition} \label{prop:projcompact} Let $C$ be a compact subspace of $\bar{L}(\Omega,\mathcal{K})$
and $C_0 \subset C \cap L(\Omega,\mathcal{K})$ a 
(possibly empty) subset such that $\bigcup_{c \in C_0} c(\Omega)$
is finite. Then there exists a continuous map
$p : C \to L(\Omega,\mathcal{K})$
such that $p(c) = c$ for all $c \in C_0$.
Moreover, $p(c)(\Omega) \subset c(\Omega)$ for all $c \in C$
and $\bigcup_{c \in C} p(c)(\Omega)$ is finite.
\end{proposition}

\begin{proof}
For any $s \in S$ and $n \in \N^* = \N \setminus \{ 0 \}$
we denote
$O_{s,n} = \{ f \in L(\Omega, S) \ | \ \la(f^{-1}(\{ s \})) > 1/n \}$. It is an open subset of $L(\Omega,S)$, hence $C_{s,n} = C \cap O_{s,n}$ is an open
subset of $C$. Now, for every $c \in C$
there exists $s \in S$ such that $\la(c^{-1}(\{ s \})) > 0$ hence $c \in C_{s,n}$ for some $n$. Then $C$ is compact and covered by the $C_{s,n}$ hence
there exists $s_1,\dots,s_r \in S$ and $n_1,\dots,n_r \in \N^*$ such that 
$C \subset \bigcup_{i=1}^r O_{s_i,n_i}$.
Up to replacing the $n_i$'s by their maximum, we may suppose $n_1 = \dots = n_r = n_0$.
Let then $F' = \bigcup_{c \in C_0} c(\Omega) \subset S$.
We set $F = \{ s_1,\dots, s_r \}\cup
 F'$. For any $i \in \{1,\dots, r \}$
we set $O_i = O_{s_i,n_0}$.

For any $c \in C$, we set $\Omega_c = \{ t \in \Omega ; c(t) \not\in F \}$, and
$$
\alpha_i(c) = \frac{d(c, \, ^c O_i)}{\sum_j d(c, ^c O_j)}
$$
and $\beta_i(c) = \sum_{k \leq i} \alpha_k(c)$, where $^c X$ denotes the complement of $X$. These define continuous maps $C \to \R_+$.
We fix an identification $\Omega \simeq [0,1]$, so
that intervals make sense inside $\Omega$.
We then set
$$
\begin{array}{lcll}
p(c)(t) &=& c(t) & \mbox{ if } c(t) \in F, \mbox{ i.e. } t \not\in \Omega_c \\
&=& s_i  & \mbox{ if } t \in \Omega_c \cap [\beta_{i-1}(c),\beta_i(c)[
\end{array}
$$

\bigskip

Let $c_1,c_2 \in C$ and $\underline{\alpha}^s$, $s = 1,2$ the corresponding
$r$-tuples $\underline{\alpha}^s = (\alpha^s_1,\dots,\alpha^s_r) \in \Delta^r$ given by $\alpha^s_i = \alpha_i(c_s)$.
When $t \not\in \Omega_{c_1} \cup \Omega_{c_2}$ 
we have $p(c_s)(t) = c_s(t)$, hence
$$
\int_{\Omega\setminus (\Omega_{c_1} \cup \Omega_{c_2})} d(p(c_1)(t), p(c_2)(t)) \dd t
\leqslant \int_{\Omega} d(c_1(t), c_2(t)) \dd t = d(c_1,c_2)
$$
and we have
$$
\int_{\Omega_{c_1} \cup \Omega_{c_2}} d(p(c_1)(t), p(c_2)(t)) \dd t \leqslant \la(\Omega_{c_1} \Delta \Omega_{c_2}) +
\int_{\Omega_{c_1} \cap \Omega_{c_2}} d(p(c_1)(t), p(c_2)(t)) \dd t.
$$
Since we know that $\la(\Omega_{c_1} \Delta \Omega_{c_2}) \leq d(c_1,c_2)$ by
Lemma \ref{lem:omegaF}, we get
$$
d(p(c_1),p(c_2)) \leq 2 d(c_1,c_2) + \int_{\Omega_{c_1} \cap \Omega_{c_2}} d(p(c_1)(t), p(c_2)(t)) \dd t
$$
and
there only remains to check that the term $\int_{\Omega_{c_1} \cap \Omega_{c_2}} d(p(c_1)(t), p(c_2)(t)) \dd t$ is
continuous. But, by Lemma \ref{lem:2rLipalphatof},
we have
$$
\int_{\Omega_{c_1} \cap \Omega_{c_2}} d(p(c_1)(t), p(c_2)(t)) \dd t
= \int_{\Omega_{c_1} \cap \Omega_{c_2}} d(f_{\underline{\alpha^1}}(t),f_{\underline{\alpha^2}}(t)) \dd t
\leq d(f_{\underline{\alpha^1}},f_{\underline{\alpha^2}}) \leq 2 r |\underline{\alpha}^1
- \underline{\alpha}^2|
$$
whence the conclusion, by continuity of $c \mapsto \underline{\alpha}$.

We must now check that $p$ takes values inside $L(\Omega,\mathcal{K})$.
Let $c \in C$. We know that $p(c)(\Omega) \subset F$ is finite, and
$$
p(c)(\Omega) \subset c(\Omega) \cup \{ s_i ; c \in O_i \}.
$$
But $c \in O_i$ implies that $s_i \in c(\Omega)$
hence  $p(c)(\Omega)$ is nonempty finite subset of $c(\Omega)$. Since $c \in \bar{L}(\Omega,\mathcal{K})$, by Lemma \ref{lem:caractfLbar} this
proves $p(c)(\Omega)\in \mathcal{K}$ and $p(c)
\in L(\Omega,\mathcal{K})$.

Finally, we have $p(c) = c$ for all $c \in C_0$ since $F \supset F'$.

\end{proof}

We immediately get the following corollary, by letting $C_0 = \{ c^0_1,\dots,c^0_k \}$.

\begin{corollary} \label{cor:projcompact} Let $C$ be a compact subset of  $\bar{L}(\Omega,\mathcal{K})$
and $c^0_1,\dots,c^0_k \in C \cap L(\Omega,\mathcal{K})$.
Then there exists a continuous map $p : C \to L(\Omega,\mathcal{K})$
such that $p(c^0_i) = c^0_i$ for all $i \in \{1,\dots, k\}$.
Moreover, $p(c)(\Omega) \subset c(\Omega)$ for all $c \in C$
and $\bigcup_{c \in C} p(c)(\Omega)$ is finite.
\end{corollary}

\bigskip

\begin{proposition} \label{prop:homotopimfinie2} Let $C$ be a compact space, and $x_0 \in C$. 
For any simplicial complex $\KK$, and any continuous map
 $\gamma : C \to L(\Omega,\KK)$, there exists a continuous map $\hat{\gamma} : (C,x_0)
\to (L(\Omega,\KK),\gamma(x_0))$ which is homotopic to $\gamma$ relative to $( \{ x_0 \}, \{\gamma(x_0) \})$,
and such that
$\bigcup_{x \in C} \hat{\gamma}(x)(\Omega)$ is finite.
\end{proposition}

\begin{proof} Let $C' = \gamma(C) \subset L(\Omega,\KK)$. It is compact, hence
applying corollary \ref{cor:projcompact} to it and to 
$\{ c_1^0 \} = \{ \gamma(x_0) \}$ we get a continuous map $p : C' \to L(\Omega,\KK)$
such that $\bigcup_{c \in C'}p(c)(\Omega)$ is finite, and $p(c)(\Omega)\subset c(\Omega)$
for all $c \in C'$. Therefore, letting $\hat{\gamma} = p \circ \gamma : C \to L(\Omega,\KK)$,
we get that $\bigcup_{x \in C} \hat{\gamma}(x)(\Omega)$ is finite.
Since $\hat{\gamma}(x)(\Omega) \subset \gamma(x)(\Omega)$ for all $x \in C$,
we get from Lemma \ref{lem:subhomotopy} that $\gamma$ and $\hat{\gamma}$ are homotopic, hence
the conclusion.

\end{proof}

\begin{proposition} Let $C$ be a compact space (and $x_0 \in C$), $\KK$
a simplicial complex, and a pair of continuous maps $\gamma_0,\gamma_1 :
C \to L(\Omega,\KK)$ (with $\gamma_0(x_0) = \gamma_1(x_0)$). If $\gamma_0$ and $\gamma_1$ are homotopic
as maps in $\LL(\Omega,\KK)$ (relative to $(\{x_0\}, \{\gamma_0(x_0) \})$), then they are homotopic inside $L(\Omega,\KK)$ (relative to $(\{x_0\}, \{\gamma_0(x_0) \})$).
\end{proposition}
\begin{proof}
After Proposition \ref{prop:homotopimfinie2}, there exists $\hat{\gamma}_0,\hat{\gamma}_1 : C \to L(\Omega,\KK)$
such that $\hat{\gamma_i}$ is homotopic to $\gamma_i$ 
with the property that $\bigcup_{x \in C} \hat{\gamma}_i(x)(\Omega)$ is finite,
for all  $i \in \{ 0, 1 \}$. Without loss of generality, one can therefore assume
that $\bigcup_{x \in C} \gamma_i(x)(\Omega)$ is finite, for all $i \in \{ 0, 1 \}$.
Let $H : C\times [0,1] \to \LL(\Omega,\KK)$ be an homotopy
between $\gamma_0$ and $\gamma_1$. Let $C' = H(C\times [0,1])$
and $C_0 = \gamma_0(C)\cup \gamma_1(C)$.  These are two compact spaces
which satisfy the assumptions of Proposition \ref{prop:projcompact}. If
 $p : C' \to L(\Omega,\KK)$ is the continuous map afforded by this proposition,
 then $\hat{H} = p \circ H$ provides a homotopy between $\gamma_0$ and $\gamma_1$
inside $L(\Omega,\KK)$. The `relative' version of the statement is proved similarly.
\end{proof}

In particular, when $C$ is equal to the $n$-sphere $S^n$, this proves
that the natural map $[S^n,L(\Omega,\KK)]_* \to 
[S^n,\LL(\Omega,\KK)]_*$ between sets of pointed homotopy classes is injective. In order to prove Theorem
\ref{theo:maintheorem} (1),
we need to prove that it is surjective.
Let us consider a continuous map $\gamma : S^n \to \LL(\Omega,\KK)$ and set $C = \gamma(S^n)$.
It is a compact subspace of $\LL(\Omega,\KK)$. Applying Proposition
\ref{prop:projcompact} with $C_0 = \emptyset$ we get $p : C \to L(\Omega,\KK)$
such that $p(c)(\Omega) \subset c(\Omega)$ for any $c \in C$. Let then $\hat{\gamma}
 = p \circ \gamma : S^n \to L(\Omega,\KK)$. From Lemma \ref{lem:subhomotopy}
we deduce that $\hat{\gamma}$ and $\gamma$ are homotopic inside $\LL(\Omega,\KK)$,
and this concludes the proof of part (1) of Theorem \ref{theo:maintheorem}.

\section{Homotopies inside $L(\Omega,\{ 0, 1 \})$}

In this section we denote $L(2) = L(\Omega,2) = L(\Omega,\{ 0 , 1 \})$, with $d(0,1) = 1$. Since we are going to use
Lipschitz properties of maps, we specify our conventions on metrics. When $(X,d_X)$ and $(Y,d_Y)$ are two metric spaces,
we endow $X \times Y$ with the metric $d_X + d_Y$, and the space $C^0([0,1],X)$ of continuous maps $[0,1] \to X$ with the metric of uniform convergence $d(\alpha,
\beta) = \| \alpha - \beta \|_{\infty} = \sup_{t \in I}|\alpha(t) - \beta(t)|$. Recall that the topology on $C^0([0,1],X)$ induced by this metric
is the compact-open topology. For short we set
$C^0(X) = C^0([0,1],X)$.

Identifying $L(2) = L(\Omega,2)$ with the space of measurable subsets of $\Omega$ (modulo subsets of measure $0$)
endowed with the metric $d(E,F) = \la(E \Delta F)$, where
$\Delta$ is the symmetric difference operator,
we have the following lemma. This lemma can
be viewed as providing a continuous  reparametrization by arc-length of natural geodesics inside the metric space $L(2)$.
\begin{lemma} \label{lem:renorm} The exists a continuous map $\mathbf{g} : L(2) \times [0,1] \to L(2)$
such that $\mathbf{g}(A,0) = A$, $\la(\mathbf{g}(A,u)) = \la(A)(1-u)$ and $\mathbf{g}(A,u) \supset \mathbf{g}(A,v)$ for all $A$ and $u \leq v$.
Moreover, it satisfies
$$\la\left(\mathbf{g}(E,u) \Delta
\mathbf{g}(F,v) \right)\leq 4 \la(E \Delta F)  + |v-u|$$ for all $E,F \in L(2)$ and $u,v \in [0,1]$.
\end{lemma}
\begin{proof}
We fix an identification $\Omega \simeq [0,1]$.
For $E \in L(2) \setminus \{ \emptyset \}$ we define $\varphi_E(t) = \la(E \cap [t,1])/\la(E)$.
The map $\varphi_E$ is obviously (weakly) decreasing and continuous $[0,1]\to [0,1]$,
with $\varphi_E(0) = 1$ and $\varphi_E(1) = 0$. It is therefore surjective, and
we can define a (weakly) decreasing map  $\psi_E : [0,1]\to[0,1]$
by $\psi_E(u) = \inf \varphi_E^{-1}(\{ u \})$. Since $\varphi_E$ is continuous,
we have $\varphi_E(\psi_E(u)) = u$.

One defines $\mathbf{g}(E,u) = E \cap [\psi_E(1-u),1]$ if $\la(E) \neq 0$,
and $\mathbf{g}(\emptyset,u) = \emptyset$. We have $\la(\mathbf{g}(E,u)) = 
\la(E \cap [\psi_E(1-u),1]) = \varphi_E(\psi_E(1-u))\la(E) = (1-u)\la(E)$
when $\la(E) \neq 0$, and $\la(\mathbf{g}(\emptyset,u)) = 0 = \la(E) (1-u)$
if $\la(E) = 0$. It is clear that $\mathbf{g}(E,u) \subset \mathbf{g}(E,v)$ for all $u \geq v$.

Moreover, clearly $\mathbf{g}(E,0) = E$
since $E \cap [\psi_E(1),1] \subset E$
and $\la(E \cap [\psi_E(1),1]) = \varphi_E(\psi_E(1)) \la(E) = \la(E)$.
It remains to prove that $\mathbf{g}$ is continuous.

Let $E,F \in L(2)$ and $u,v \in [0,1]$. We first assume $\la(E) \la(F) > 0$.
Without loss of generality we can assume $\psi_E(1-u) \leq \psi_F(1-v)$.
Then $[\psi_E(1-u),1] \supset [\psi_F(1-v),1]$, and $\mathbf{g}(E,u) \Delta
\mathbf{g}(F,v)$ can be decomposed as
$$
\left( (E \setminus F) \cap [\psi_E(1-u),1] \right)
\cup
\left( (F \setminus E) \cap [\psi_F(1-v),1] \right)
\cup
\left( (E \cap F) \cap [\psi_E(1-u),\psi_F(1-v)] \right).
$$
Since the first two pieces are included inside $E \Delta F$,
we get
$\la(\mathbf{g}(E,u) \Delta
\mathbf{g}(F,v)) \leq \la(E \Delta F) + \la\left( (E \cap F) \cap [\psi_E(1-u),\psi_F(1-v)] \right)$.
Now $ (E \cap F) \cap [\psi_E(1-u),\psi_F(1-v)] =
 \left(E \cap F \cap [\psi_E(1-u),1] \right) \setminus
  \left(E \cap F \cap [\psi_F(1-v),1] \right)$
hence 
$$
\begin{array}{lcl}
\la \left( (E \cap F) \cap [\psi_E(1-u),\psi_F(1-v)] \right)
&=& \la \left(E \cap F \cap [\psi_E(1-u),1] \right) - \la \left(E \cap F \cap [\psi_F(1-v),1] \right) \\
&\leq& \la \left(E  \cap [\psi_E(1-u),1] \right) - \la \left(E \cap F \cap [\psi_F(1-v),1] \right) \\
&\leq& (1-u)\la(E) - \la \left(E \cap F \cap [\psi_F(1-v),1] \right)
\end{array}
$$
Now, since $F  = (E \cap F) \sqcup (F \setminus E)$,
we have $F \cap [\psi_F(1-v),1] = \left( (E \cap F) \cap [\psi_F(1-v),1] \right) \sqcup
\left( (F \setminus E) \cap [\psi_F(1-v),1] \right)$ hence
$$
\begin{array}{lcl}
(1-v) \la(F) & = & \la \left( (E \cap F) \cap [\psi_F(1-v),1] \right) + \la \left( (F \setminus E) \cap [\psi_F(1-v),1] \right) \\
&\leq & \la \left( (E \cap F) \cap [\psi_F(1-v),1] \right) + \la (F \setminus E)  \\
&\leq & \la \left( (E \cap F) \cap [\psi_F(1-v),1] \right) + \la (F \Delta E).
\end{array}
$$
It follows that $-\la \left( (E \cap F) \cap [\psi_F(1-v),1] \right) \leq \la(F \Delta E) - (1-v)\la(F)$
hence
$$
\la \left( (E \cap F) \cap [\psi_E(1-u),\psi_F(1-v)] \right) \leqslant (1-u)\la(E) + \la(F \Delta E) - (1-v)\la(F)
$$
and finally 
$$
\begin{array}{lcl}
\la(\mathbf{g}(E,u) \Delta
\mathbf{g}(F,v)) 
& \leq & 2 \la(E \Delta F) + (1-u)\la(E) - (1-v)\la(F) \\
& \leq & 2 \la(E \Delta F) +(\la(E)-\la(F)) + (v-u) \la(E) + v (\la(F)- \la(E)) \\
& \leq & 2 \la(E \Delta F) +|\la(E)-\la(F)| + |v-u| \la(E) + v |\la(F)- \la(E)| \\
& \leq & 2 \la(E \Delta F) +2 |\la(E)-\la(F)| + |v-u|   \\
& \leq & 4 \la(E \Delta F)  + |v-u|.  \\
\end{array}
$$
Therefore we get the inequality $\la(\mathbf{g}(E,u) \Delta
\mathbf{g}(F,v)) \leq 4 \la(E \Delta F)  + |v-u|$,
that we readily check to hold also when $\la(E) \la(F) = 0$.
This proves that $\mathbf{g}$ is continuous, whence the claim.
\end{proof}

We provide a 2-dimensional illustration, with $\Omega = [0,1]^2$.  
The map constructed in the proof depends on an identification $[0,1]^2 \simeq [0,1]$ (up to a set of measure $0$). An explicit one is given
by the binary-digit identification
$$
0.\eps_1\eps_2\eps_3\dots \mapsto (0.\eps_1\eps_3\eps_5\dots, 0.\eps_2\eps_4\eps_6\dots)
$$
with the $\eps_i \in \{0,1\}$.
Then, when $A$ is some (blue) rectangle, the map $u \mapsto \mathbf{g}(A,u)$ looks as follows.

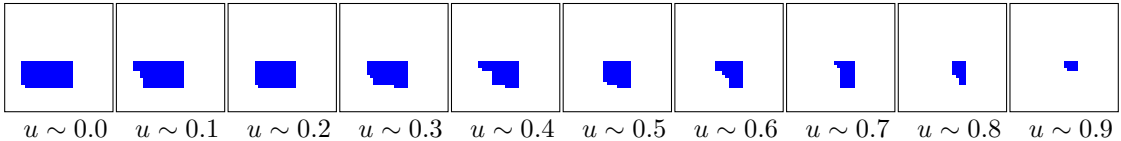
\begin{center}
\input{figgg2.tex}
\end{center}

The above lemma is actually all what is needed to prove Theorem \ref{theo:maintheorem}
in the case of binary random variables, that is $S = \{ 0 ,1 \}$, as we will illustrate later (see
corollary \ref{cor:Hurewics01}). In the general case however, we shall need
a more powerful homotopy, provided by Proposition \ref{prop:phi} below. The next lemmas are preliminary
technical steps in view of its proof.

\begin{lemma} \label{lem:EAtoalpha} The map $C^0(L(2)) \times L(2) \to C^0([0,1])$
defined by $(E_{\bullet},A) \to \alpha$
where $\alpha(u) = \la(E_u \cap A)$, is 1-Lipschitz.
\end{lemma}
\begin{proof} Let  $\alpha$, $\beta$ denote the images of 
$(E_{\bullet},A)$ and $(F_{\bullet},B)$, respectively.
Then, for all $u \in I$,
we have
$$
|\alpha(u)-\beta(u) | = |\la(E_u \cap A) - \la(F_u \cap B)| \leq \la\left((E_u \cap A) \Delta (F_u \cap B)\right)
$$
From the general set-theoretic inequality $(X \cap A) \Delta (Y \cap B) \subset (X \Delta Y) \cup (A \Delta B)$
one gets 
$$\la\left((E_u \cap A) \Delta (F_u \cap B)\right) \leq \la( E_u \Delta F_u) + \la(A \Delta B),
$$
hence $\|\alpha - \beta\|_{\infty} \leq \sup_u \la(E_u \Delta F_u) + \la(A \Delta B)$ and this proves
the claim.
\end{proof}

\begin{lemma} \label{lem:phimoins}
A map $\Phi_{-} : C^0([0,1]) \times C^0(L(2)) \times L(2) \to C^0(L(2))$ is defined as follows.
To $(a,E_{\bullet},A) \in C^0([0,1]) \times C^0(L(2)) \times L(2) \to C^0(L(2))$
one associates the map
$$
\Phi_{-}(a,E_{\bullet},A) : u  \mapsto \mathbf{g}\left( E_u \cap A , 1 - \frac{\min(a(u)\la(E_u),\alpha(u))}{\alpha(u)} \right)
$$
if $\alpha(u) \neq 0$, and otherwise $u \mapsto \emptyset$,
where $\alpha(u) = \la(A \cap E_u)$.
Then, the map $\Phi_{-}$ is continuous.
\end{lemma}

\begin{proof}
Let us fix $(a,E_{\bullet},A) \in C^0([0,1]) \times C^0(L(2)) \times L(2) $, and let $\eps > 0$.
Consider $\hat{m} : [0,1]\times [\eps/12,1] \to [0,1]$
be defined by $\hat{m}(x,y) = \min(x,y)/y$. It is clearly continuous on the compact space $[0,1]\times [\eps/12,1]$,
hence unformly continuous, hence there exists $\eta > 0$ such that $\max(|x_1 -x_2|,|y_1 - y_2|) < \eta \Rightarrow
|\hat{m}(x_1,y_1) - \hat{m}(x_2,y_2)| \leq \eps/6$. Clearly one can assume $\eta \leq \eps/6$ as well.

Let us then consider $(b,F_{\bullet},B) \in C^0([0,1]) \times C^0(L(2)) \times L(2) $
such that $\| a - b \|_{\infty} + \sup_u \la(E_u\Delta F_u) + \la(A \Delta B) \leq \eta$.
From Lemma \ref{lem:EAtoalpha},  we get $\| \alpha - \beta \|_{\infty} \leq \eta$.
Let us consider $I_0 = \{ u \in [0,1] \ | \ \alpha(u) \leq \eps/3 \}$. We have by definition
$\alpha([0,1] \setminus I_0) \subset ]\eps/3,1] \subset [\eps/12,1]$
and, since $\| \alpha - \beta \|_{\infty} \leq \eps/6$,
we have 
$\beta([0,1] \setminus I_0) \subset ]\eps/6,1] \subset [\eps/12,1]$.
Moreover, since
$$
\begin{array}{lclcl}
|a(u)\la(E_u) - b(u)\la(F_u)| &\leq& |a(u) - b(u)|\la(E_u)  + b(u)|\la(E_u) - \la(F_u)| \\
&\leq &|a(u) - b(u)| + \la(E_u \Delta F_u) &\leq& \eta \\
\end{array} 
$$
we get that, for all $u \not\in I_0$, we have
$|\hat{m}(a(u)\la(E_u)),\alpha(u)) - \hat{m}(b(u)\la(F_u),\beta(u)) | \leq \eps/6$.
Moreover, since in particular $\alpha(u) \beta(u) \neq 0$,
we get from the general inequality $\la(\mathbf{g}(X,x) \Delta \mathbf{g}(Y,y)) \leq
4 \la(X \Delta Y) + |x-y|$ of Lemma \ref{lem:renorm} that,
for all $u \not\in I_0$,
$$
\begin{array}{lcl}
d\left(\Phi_{-}(a,E_{\bullet},A)(u), \Phi_{-}(b,F_{\bullet},B)(u)\right) &\leq& 4 \la((E_u \cap A) \Delta (F_u \cap B))
\\ && + |\hat{m}(a(u)\la(E_u),\alpha(u)) - \hat{m}(b(u)\la(F_u),\beta(u)) | \\
& \leq & 4 \left( \la(E_u \Delta F_u) +  \la (A \Delta B) \right) + \eps/6\\
& \leq & 4 \eps/6 + \eps/6\\
& < & \eps
\end{array}
$$
Now, if $u \in I_0$, then $\Phi_{-}(a,E_{\bullet},A)(u) \subset E_u \cap A$
hence $\la(\Phi_{-}(a,E_{\bullet},A)(u)) \leq \la(E_u \cap A) = \alpha(u) \leq \eps/3$
and $\la(\Phi_{-}(b,F_{\bullet},B)(u)) \leq \la(F_u \cap B) = \beta(u) \leq \eps/3 + \eps/6 = \eps/2$,
whence $$d\left(\Phi_{-}(a,E_{\bullet},A)(u), \Phi_{-}(b,F_{\bullet},B)(u)\right) \leqslant
\la(\Phi_{-}(a,E_{\bullet},A)(u)) + \la(\Phi_{-}(b,F_{\bullet},B)(u)) \leqslant 5 \eps/6 < \eps.$$
It follows that $d\left(\Phi_{-}(a,E_{\bullet},A), \Phi_{-}(b,F_{\bullet},B)\right) \leq \eps$
and $\Phi_{-}$ is continuous at $(a,E_{\bullet},A)$, which proves the claim.

\end{proof}

We use the convention $\mathbf{g}(X,t) = X$ for $t \leq 0$ and $\mathbf{g}(X,t) = \emptyset$
for $t > 1$, so that $\mathbf{g}$ is extended to a continuous map $L(2) \times \R \to L(2)$. The notation $^c A$ denotes the complement inside $\Omega$ of the set $A$, identified with an element of $L(\Omega,2)$.

\begin{lemma} \label{lem:phiplus} 
A map $\Phi_{+} : C^0([0,1]) \times C^0(L(2)) \times L(2) \to C^0(L(2))$ is defined as follows.
To $(a,E_{\bullet},A) \in C^0([0,1]) \times C^0(L(2)) \times L(2) \to C^0(L(2))$
one associates the map
$$
\Phi_{+}(a,E_{\bullet},A) : u  \mapsto \mathbf{g}\left( E_u \cap ( ^c A) , 1 - \frac{\max(0,a(u)\la(E_u)-\alpha(u))}{\la(E_u) - \alpha(u)} \right)
$$
if $\alpha(u) \neq \la(E_u)$, and otherwise $u \mapsto \emptyset$,
where $\alpha(u) = \la(A \cap E_u)$.
Then, the map $\Phi_{+}$ is continuous.
\end{lemma}

The proof is similar to the one of the previous lemma, and left to the reader.

\begin{lemma} \label{lem:lipcup} The map $(f,g) \mapsto (t \mapsto f(t) \cup g(t))$
is continuous $C^0(L(2))^2 \to C^0(L(2))$, and even 1-Lipschitz.
\end{lemma}
\begin{proof}
The map $(X,Y) \mapsto X \cup Y$ is 1-Lipschitz because of the general
set-theoretic fact 
$(X_1 \cup Y_1) \Delta (X_2 \cup Y_2) \subset (X_1 \Delta X_2) \cup (Y_1 \Delta Y_2)$
from which we deduce
$\la((X_1 \cup Y_1) \Delta (X_2 \cup Y_2)) \leq \la(X_1 \Delta X_2) + \la(Y_1 \Delta Y_2)$,
which proves that $(X,Y) \mapsto X \cup Y$ is 1-Lipschitz $L(2)^2 \to L(2)$. It
follows that the induced map $C^0(L(2)^2) = C^0(L(2))^2 \to C^0(L(2))$
is 1-Lipschitz and thus continuous, too.
\end{proof}

The following proposition informally says that, when $E_{\bullet} \in C^0(L(2))$ is a path inside $L(2)$ with $A \subset E_0$, then we can find another path $\Phi_{\bullet} \in C^0(L(2))$ such that $\Phi_u \subset E_u$ for all $u$, and the ratio $\la(\Phi_{\bullet})/\la(E_{\bullet})$ follows any previously specified variation starting at $\la(A)/\la(E_0)$ -- and, moreover, that this can be done continuously.

\begin{proposition} \label{prop:phi} There exists a continuous map
$\Phi : C^0([0,1]) \times C^0(L(2)) \times L(2) \to C^0(L(2))$ 
having the following 
properties.
\begin{itemize}
\item for all $(a,E_{\bullet},A) \in C^0([0,1]) \times C^0(L(2)) \times L(2)$ such that
$A \subset E_0$ and $a(0)\la(E_0)  = \la(A)$, we have
$\Phi(a,E_{\bullet},A)(0) = A$
\item for all $u \in [0,1]$,
$\Phi(a,E_{\bullet},A)(u) \subset E_u$ and
$\la(\Phi(a,E_{\bullet},A)(u)) = a(u)\la(E_u)$
\item if $a$ and $E_{\bullet}$ are constant maps, then so is $\Phi(a,E_{\bullet},A)$.
\end{itemize}

\end{proposition}
\begin{proof}

We define $\Phi(a,E_{\bullet},A)(u) = \Phi_{-}(a,E_{\bullet},A)(u) \cup \Phi_{+}(a,E_{\bullet},A)(u)$.
By the definition of $\Phi_{\pm}$ in Lemmas \ref{lem:phimoins} and \ref{lem:phiplus}, the last property is clear.
By combining Lemmas \ref{lem:phimoins}, \ref{lem:phiplus} and \ref{lem:lipcup} we get that
$\Phi$ is continuous. Moreover, $\Phi_{-}(a,E_{\bullet},A)(u) \subset E_u \cap A$
and $\Phi_{+}(a,E_{\bullet},A)(u) \subset E_u \cap (^cA)$
hence $\Phi(a,E_{\bullet},A)(u) = \Phi_{-}(a,E_{\bullet},A)(u) \sqcup \Phi_{+}(a,E_{\bullet},A)(u) \subset E_u$,
with $\la(\Phi(a,E_{\bullet},A)(u)) = \la(\Phi_{-}(a,E_{\bullet},A)(u)) + \la(\Phi_{+}(a,E_{\bullet},A)(u))$.
Letting $\alpha(u) = \la(E_u \cap A)$, again by Lemmas \ref{lem:phimoins} and \ref{lem:phiplus} we get
$$\la(\Phi_{-}(a,E_{\bullet},A)(u)) = \la\left( \mathbf{g}\left( E_u \cap A , 1 - \frac{\min(a(u)\la(E_u),\alpha(u))}{\alpha(u)} \right)\right)
=\min(a(u)\la(E_u),\alpha(u))
$$
and, since $\la(E_u) - \alpha(u) = \la(E_u) - \la(A \cap E_u) = \la( ( ^c A) \cap E_u)$, 
$$\la(\Phi_{+}(a,E_{\bullet},A)(u)) = \la\left( \mathbf{g}\left( E_u \cap (^cA) , 1 - \frac{\max(0,a(u)\la(E_u)-\alpha(u))}{\la((^c A) \cap E_u) } \right)\right)
= \max(0,a(u)\la(E_u)-\alpha(u)) .
$$
Therefore we get $\la(\Phi(a,E_{\bullet},A)(u)) = \max(0,a(u)\la(E_u)-\alpha(u)) + \min(a(u)\la(E_u),\alpha(u)) = a(u)\la(E_u)$ for all $u \in [0,1]$. Finally, since $A \subset E_0$ and 
$\alpha(0) = \la(E_0 \cap A) = \la(A) = \la(E_0)a(0)$,
we get that $\Phi(a,E_{\bullet},A)(0) = \mathbf{g}(E_0 \cap A,0)\cup \mathbf{g}(E_0 \cap (^c A), 1) = A \cup \emptyset = A$,
and this proves the claim.
\end{proof}

As before, we provide an illustration, when $A \subset \Omega$ is the same (blue) rectangle, and $E_{\bullet}$ associates continuously to any $u \in [0,1]$ some rectangle, whose boundary is dashed and in red. In this example, the map $a$ is taken to be affine, from $\la(A)/\la(E_0)$ to $0$. The first row depicts the map $u \mapsto E_u$, and the second row superposes it with
the map $u \mapsto \Phi(a,E_{\bullet},A)(u)$, depicted
in blue.

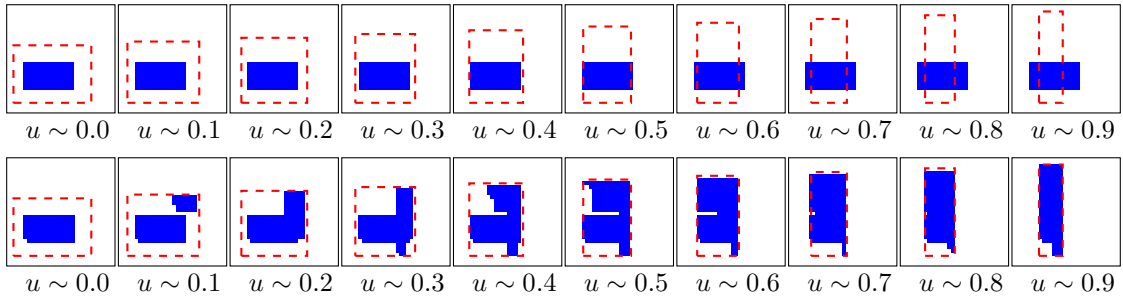
\begin{center}
\begin{tikzpicture}[scale=.45]
\begin{scope}[scale=.1,shift={(0,0)}]
\draw (0,0) -- (32,0) -- (32,32) -- (0,32) -- cycle;
\fill[blue] (5,7) -- (5,15) -- (20,15) -- (20,7) -- cycle;
\draw (18,-5) node {$u \sim 0.0$};
\draw[dashed,thick,red]  (2,3) -- (2,20) -- (25,20) -- (25,3) -- cycle;
\end{scope}
\begin{scope}[scale=.1,shift={(33,0)}]
\draw (0,0) -- (32,0) -- (32,32) -- (0,32) -- cycle;
\fill[blue] (5,7) -- (5,15) -- (20,15) -- (20,7) -- cycle;
\draw (18,-5) node {$u \sim 0.1$};
\draw[dashed,thick,red]  (8/3,3) -- (8/3,190/9) -- (215/9,190/9) -- (215/9,3) -- cycle;
\end{scope}
\begin{scope}[scale=.1,shift={(66,0)}]
\draw (0,0) -- (32,0) -- (32,32) -- (0,32) -- cycle;
\fill[blue] (5,7) -- (5,15) -- (20,15) -- (20,7) -- cycle;
\draw (18,-5) node {$u \sim 0.2$};
\draw[dashed,thick,red]  (10/3,3) -- (10/3,200/9) -- (205/9,200/9) -- (205/9,3) -- cycle;
\end{scope}
\begin{scope}[scale=.1,shift={(99,0)}]
\draw (0,0) -- (32,0) -- (32,32) -- (0,32) -- cycle;
\fill[blue] (5,7) -- (5,15) -- (20,15) -- (20,7) -- cycle;
\draw (18,-5) node {$u \sim 0.3$};
\draw[dashed,thick,red]  (4,3) -- (4,70/3) -- (65/3,70/3) -- (65/3,3) -- cycle;
\end{scope}
\begin{scope}[scale=.1,shift={(132,0)}]
\draw (0,0) -- (32,0) -- (32,32) -- (0,32) -- cycle;
\fill[blue] (5,7) -- (5,15) -- (20,15) -- (20,7) -- cycle;
\draw (18,-5) node {$u \sim 0.4$};
\draw[dashed,thick,red]  (14/3,3) -- (14/3,220/9) -- (185/9,220/9) -- (185/9,3) -- cycle;
\end{scope}
\begin{scope}[scale=.1,shift={(165,0)}]
\draw (0,0) -- (32,0) -- (32,32) -- (0,32) -- cycle;
\fill[blue] (5,7) -- (5,15) -- (20,15) -- (20,7) -- cycle;
\draw (18,-5) node {$u \sim 0.5$};
\draw[dashed,thick,red]  (16/3,3) -- (16/3,230/9) -- (175/9,230/9) -- (175/9,3) -- cycle;
\end{scope}
\begin{scope}[scale=.1,shift={(198,0)}]
\draw (0,0) -- (32,0) -- (32,32) -- (0,32) -- cycle;
\fill[blue] (5,7) -- (5,15) -- (20,15) -- (20,7) -- cycle;
\draw (18,-5) node {$u \sim 0.6$};
\draw[dashed,thick,red]  (6,3) -- (6,80/3) -- (55/3,80/3) -- (55/3,3) -- cycle;
\end{scope}
\begin{scope}[scale=.1,shift={(231,0)}]
\draw (0,0) -- (32,0) -- (32,32) -- (0,32) -- cycle;
\fill[blue] (5,7) -- (5,15) -- (20,15) -- (20,7) -- cycle;
\draw (18,-5) node {$u \sim 0.7$};
\draw[dashed,thick,red]  (20/3,3) -- (20/3,250/9) -- (155/9,250/9) -- (155/9,3) -- cycle;
\end{scope}
\begin{scope}[scale=.1,shift={(264,0)}]
\draw (0,0) -- (32,0) -- (32,32) -- (0,32) -- cycle;
\fill[blue] (5,7) -- (5,15) -- (20,15) -- (20,7) -- cycle;
\draw (18,-5) node {$u \sim 0.8$};
\draw[dashed,thick,red]  (22/3,3) -- (22/3,260/9) -- (145/9,260/9) -- (145/9,3) -- cycle;
\end{scope}
\begin{scope}[scale=.1,shift={(297,0)}]
\draw (0,0) -- (32,0) -- (32,32) -- (0,32) -- cycle;
\fill[blue] (5,7) -- (5,15) -- (20,15) -- (20,7) -- cycle;
\draw (18,-5) node {$u \sim 0.9$};
\draw[dashed,thick,red]  (8,3) -- (8,30) -- (15,30) -- (15,3) -- cycle;
\end{scope}
\end{tikzpicture}

\end{center}
\begin{center}
\input{figph2.tex}
\end{center}

\section{Probability law}

\subsection{The law maps}
Recall from \cite{SPANIER} that the weak (or coherent) topology on
$|\mathcal{K}|$ is the topology such that $U$ is open in $|\mathcal{K}|$
iff $U \cap |F|$ is open for every $F \in \mathcal{K}$,
where $|F| = \{ \alpha : F \to [0,1] \ | \ \sum_{s \in F} \alpha(s) = 1 \}$
is given the topology induced from the product topology of $[0,1]^F$. 
For each $p \geq 1$, we can put a metric topology on the same set,
in order to define a metric space $|\KK|_{d_p}$ by
the metric $d_p(\alpha,\beta) = \sqrt[p]{\sum_{s \in S}
|\alpha(s)-\beta(s)|^p}$. The map $|\KK| \to |\KK|_{d_p}$ is continuous, and it is an homeomorphism iff $|\KK|$ is metrizable iff it is satisfies the first axiom of countability,
iff $\mathcal{K}$ is locally finite (see \cite{SPANIER} p. 119 ch. 3 sec. 2 Theorem 8 for the case $p=2$, but the proof works for $p\neq 2$ as well).

For $\alpha : S \to [0,1]$, we denote the \emph{support} of $\alpha$ by $\supp(\alpha) = \{ s \in S \ | \ \alpha(s) \neq 0 \}$. We let $\Psi_0 : L(\Omega,\KK) \to |\KK|$ be defined by
associating to a random variable $f \in L(\Omega,\KK)$ its
probability law $s \mapsto \la( f^{-1}(\{s\}))$.

\subsection{Non-continuity of $\Psi_0$}
We first prove that $\Psi_0$ is \emph{not} continous in general, by providing an example.
Let us consider $S = \N = \Z_{\geq 0}$, and $\KK = \PF^*(\N)$. We introduce
$$
U = \left\lbrace \alpha \in |\KK| \ ; \ \forall s \neq 0  \  \ \alpha(s) < \frac{1}{\# \supp(\alpha)} \right\rbrace.
$$
We note that $U$ is open in $|\KK|$. Indeed, if $F \in \KK$
we have 
$$
U \cap |F|  = 
 \left\lbrace \alpha : F \to [0,1] \ | \ \sum_{s \in F} \alpha(s) = 1 \ \& \ \forall s \neq 0  \ | \ \alpha(s) < \frac{1}{\# \supp(\alpha)} \right\rbrace
$$
which is equal to 
$$
 \bigcup_{G \subset F \setminus \{ 0 \}}
 \left\lbrace \alpha : G \to [0,1] \ | \ \alpha(0) + \sum_{s \in G} \alpha(s) = 1 \ \& \ \forall s \in G  \ | \ 0< \alpha(s) < \frac{1}{\# G +1} \right\rbrace
$$
 and it is open as the union of a finite collection of open sets.
 Now consider $\Psi_0^{-1}(U)$, and let $f_0\in L(\Omega,\KK)$ be the constant map $t \mapsto 0$. Clearly $\alpha_0 = \Psi_0(f_0)$ is the map $0 \mapsto 1$, $k \mapsto 0$ for $k \geq 1$, and $\alpha_0 \in U$. If $\Psi_0^{-1}(U)$ is open, there exists $\eps>0$ such
 that it contains the open ball centered at $f_0$ with radius $\eps$. Let $n$ be such that $1/n < \eps/3$, and define $f \in L([0,1],\KK)$ by $f(t) = 0$ for $t \in [0,1-2/n[$,
 $f(t) = k$ for $t \in [1-\frac{2}{n} + \frac{k-1}{n^3},1-\frac{2}{n} + \frac{k}{n^3}[$ and $1 \leq k \leq n^2$, 
and finally $f(t) = n^2+1$ for $t \in [1-\frac{1}{n},1]$.
The graph of $f$ for $n = 3$ is depicted below.
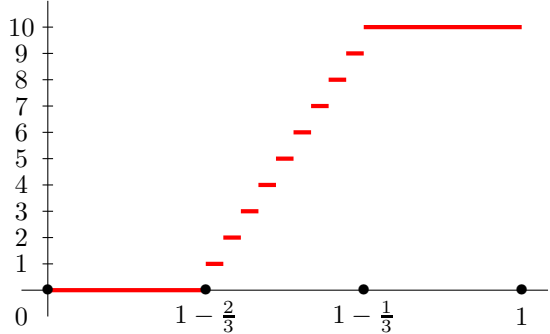
\begin{center}
\begin{tikzpicture}[scale=.7]
\draw (-0.5,0) -- (9.5,0);
\draw (0,-0.5) -- (0,5.5);
\draw[red,ultra thick] (0,0) -- (3,0);
\draw (-0.5,-0.5) node {$0$};
\draw (0,0) node {$\bullet$};
\draw (3,0) node {$\bullet$};
\draw (6,0) node {$\bullet$};
\draw (9,0) node {$\bullet$};
\draw (3,-0.5) node {$1 - \frac{2}{3}$};
\draw (6,-0.5) node {$1 - \frac{1}{3}$};
\draw (9,-0.5) node {$1$};
\draw (-0.1,0.5*1) -- (0.1,0.5*1);
\draw (-0.1,0.5*2) -- (0.1,0.5*2);
\draw (-0.1,0.5*3) -- (0.1,0.5*3);
\draw (-0.1,0.5*4) -- (0.1,0.5*4);
\draw (-0.1,0.5*5) -- (0.1,0.5*5);
\draw (-0.1,0.5*6) -- (0.1,0.5*6);
\draw (-0.1,0.5*7) -- (0.1,0.5*7);
\draw (-0.1,0.5*8) -- (0.1,0.5*8);
\draw (-0.1,0.5*9) -- (0.1,0.5*9);
\draw (-0.1,0.5*10) -- (0.1,0.5*10);
\draw (-0.5,0.5*1) node {$1$};
\draw (-0.5,0.5*2) node {$2$};
\draw (-0.5,0.5*3) node {$3$};
\draw (-0.5,0.5*4) node {$4$};
\draw (-0.5,0.5*5) node {$5$};
\draw (-0.5,0.5*6) node {$6$};
\draw (-0.5,0.5*7) node {$7$};
\draw (-0.5,0.5*8) node {$8$};
\draw (-0.5,0.5*9) node {$9$};
\draw (-0.5,0.5*10) node {$10$};
\draw[red,ultra thick] (6,5) -- (9,5);
\draw[red, ultra thick] (3+0,0.5*1) -- (3+1/3,0.5*1);
\draw[red, ultra thick] (3+1/3,0.5*2) -- (3+2/3,0.5*2);
\draw[red, ultra thick] (3+2/3,0.5*3) -- (3+1,0.5*3);
\draw[red, ultra thick] (3+1,0.5*4) -- (3+4/3,0.5*4);
\draw[red, ultra thick] (3+4/3,0.5*5) -- (3+5/3,0.5*5);
\draw[red, ultra thick] (3+5/3,0.5*6) -- (3+2,0.5*6);
\draw[red, ultra thick] (3+2,0.5*7) -- (3+7/3,0.5*7);
\draw[red, ultra thick] (3+7/3,0.5*8) -- (3+8/3,0.5*8);
\draw[red, ultra thick] (3+8/3,0.5*9) -- (3+3,0.5*9);

\end{tikzpicture}
\end{center} 

We have $d(f,f_0) = 2/n < 2 \eps/3 < \eps$ hence we should have
$\alpha = \Psi_0(f) \in U$. But the support of $\alpha$ has cardinality $n^2+2$,
and $\alpha(n^2+1) = 1/n > 1/(n^2+2)$, contradicting $\alpha \in U$. This proves that $\Psi_0$ is not continuous.

\subsection{Continuity of $\Psi$ and existence of global sections}

For short, we now denote $|\KK|_p = |\KK|_{d_p}$. We consider the same `law' map
$\Psi : L(\Omega,\KK) \to |\KK|_1$. We prove that it is uniformly continuous (and actually 2-Lipschitz).
Indeed, if $f,g \in L(\Omega,\KK)$, and $\alpha = \Psi(f)$, $\beta = \Psi(g)$,
then 
$$
d_1(\alpha,\beta) = \sum_{s \in S} |\alpha(s) - \beta(s)|
= \sum_{s \in S} |\la(f^{-1}(s)) - \la(g^{-1}(s))|
$$
and $|\la(f^{-1}(s)) - \la(g^{-1}(s))| \leq \la(f^{-1}(s) \Delta g^{-1}(s))$.
But $f^{-1}(s) \Delta g^{-1}(s) = \{ t \in f^{-1}(s) \ | \ f(t) \neq g(t) \} \cup 
\{ t \in g^{-1}(s) \ | \ f(t) \neq g(t) \}$
whence
$$
d_1(\alpha,\beta) \leqslant \sum_{s \in S} \int_{f^{-1}(s)} d(f(t),g(t))\dd t
+ \sum_{s \in S} \int_{g^{-1}(s)} d(f(t),g(t))\dd t = 2 \int_{\Omega} d(f(t),g(t))\dd t
$$
whence $d_1(\alpha,\beta) \leq 2 d(f,g)$. It follows that it
induces a continuous map $\bar{L}(\Omega,\KK) \to \overline{|\KK|_1}$,
where $$\overline{|\KK|_1} = \{ \alpha : S \to [0,1] \ | \ \PF^*(\supp(\alpha))\subset \KK \ \& \ \sum_{s \in S} \alpha(s)  = 1 \}$$
endowed with the metric $d(\alpha,\beta) = \sum_{s \in S} |\alpha(s)-\beta(s)|$
is the completion of $|\KK|_1$.
This map associates to $f \in \LL(\Omega,\KK)$ the map $\alpha(s) = \la(f^{-1}(s))$.
Notice that the condition $\sum_s \alpha(s) =1 < \infty$ implies that the support $\supp(\alpha)$
of $\alpha$ is finite.

The fact that $|\KK|_1$ has the same homotopy type than $|\KK|$
has originally been proved by Dowker in \cite{DOWKER} in a more general
context, and another proof was subsequently provided by Milnor in \cite{CWMILNOR}. 

It is clear that every mass distribution on the discrete set $S$ is realizable by some
random variable. We first show that it is possible to do this \emph{continuously}. In topological terms,
this proves the following statement.

\begin{proposition} The maps $\Psi$ and $\overline{\Psi}$ admit global (continuous) sections.
\end{proposition}
\begin{proof}
We fix some (total) ordering $\leq$ on $S$
and some identification $\Omega \simeq [0,1]$.
We define $\sigma : \overline{|\KK|_1} \to \LL(\Omega,\KK)$ as follows. For any $\alpha \in \overline{|\KK|_1}$, $S_{\alpha} = \supp(\alpha) \subset S$ is countable. Let $A_{\pm} : S \to \R_+$ denote the associated cumulative mass functions $A_+(s) = \sum_{u \leq s} \alpha(u)$
and $A_-(s) = \sum_{u < s} \alpha(u)$. They induce increasing injections $(S_{\alpha}, \leq) \to [0,1]$. The map
$\sigma(\alpha)$ is defined by $\sigma(\alpha)(t) =a$ if $A_-(a) \leq t < A_+(a)$. 
We have $\sigma(\alpha)(\Omega) = S_{\alpha}$. Since $\alpha \in \overline{|\KK|}_1$
every non-empty finite subset of $S_{\alpha}$ belongs to $\KK$ hence $\sigma(\alpha) \in \bar{L}(\Omega,\KK)$, and $\sigma(\alpha) \in L(\Omega,\KK)$ as soon as $\alpha \in |\KK|_1$.

Clearly $\bar{\Psi} \circ \sigma$ is the identity.
We prove now that $\sigma$
is continuous at any $\alpha \in \overline{|\KK|_1}$. Let $\eps>0$. There exists $S_{\alpha}^0 \subset S_{\alpha}$ finite (and non-empty) such that
$\sum_{s \in S_{\alpha} \setminus S_{\alpha}^0} \alpha(s) \leq \eps/3$. Let $n = |S_{\alpha}^0| > 0$.
 We set $\eta = \eps/3n$. Let $\beta \in \overline{|\KK|_1}$ with $|\alpha - \beta|_1 \leq \eta$,
 and set $B_+(s) = \sum_{u \leq s} \beta(u)$
and $B_-(s) = \sum_{u < s} \beta(u)$.
We have 
$$
d(\sigma(\alpha),\sigma(\beta)) \leq \eps/3+ \sum_{a \in S_{\alpha}^0} \int_{A_-(a)}^{A^+(a)} d(\sigma(\alpha)(t),\sigma(\beta)(t)) \dd t
$$
Now note that $|A_{\pm}(a)-B_{\pm}(a)| \leq |\alpha - \beta|_1 \leq \eps/3n$ for each $a \in S_{\alpha}^0$ hence
$$\int_{A_-(a)}^{A^+(a)} d(\sigma(\alpha)(t),\sigma(\beta)(t)) \dd t \leq \frac{2\eps}{3n} + \int_{\max(A_-(a),B_-(a))}^{\min(A_+(a),B_+(a))} d(\sigma(\alpha)(t),\sigma(\beta)(t)) \dd t = \frac{2\eps}{3n}
$$
since $\sigma(\alpha)(t) = \sigma(\beta)(t)$ for each $t \in [\max(A_-(a),B_-(a)), \min(A_+(a),B_+(a))]$,
and this yields $d(\sigma(\alpha),\sigma(\beta)) \leq \eps$. This proves that $\sigma$ is continuous at any $\alpha \in \bar{L}(\Omega,\KK)$.
Therefore $\sigma$ provides a continuous global section of $\bar{\Psi}$, which obviously restricts to a continuous global
section of $\Psi$.
\end{proof}

\subsection{Homotopy lifting properties}
Let $\Psi_{\KK} : L(\Omega,\KK) \to |\KK|_1$
and $\bar{\Psi}_{\KK} : \LL(\Omega,\KK) \to |\KK|_1$ denote the law map. 
If $\alpha$ is a cardinal, we let $\Psi_{\alpha}$ (resp. $\bar{\Psi}_{\alpha}$) denote the map associated to
the simplicial complex $\PF^*(\alpha)$. Recall that a continuous map $p : E \to B$ is said to have the homotopy lifting property
(HLP) with respect to some topological space $X$ if, for any (continuous) maps $H : X \times [0,1] \to B$ and $h : X \to E$ 
such that $p \circ h = H(\bullet,0)$, there exists a map $\tilde{H} = X \times [0,1] \to E$ such that $p \circ \tilde{H} = H$
and $\tilde{H}(\bullet,0) = h$.
$$
\xymatrix{
 & E \ar[d]_p & & E \ar[d]_p \\
X \ar[ur]^h \ar[r]_{p \circ h} & B & X \times [0,1] \ar[r]_H \ar@{.>}[ur]^{\tilde{H}} & B
}
$$

A Hurewicz fibration
is a map having the HLP w.r.t. arbitrary topological spaces. A Serre fibration
is a map having the HLP w.r.t. all $n$-spheres, and this is equivalent to having the HLP w.r.t. any CW-complex.

\begin{lemma} \label{lem:redcardinal} If $\Psi_{\alpha}$ (resp. $\bar{\Psi}_{\alpha}$) has the HLP w.r.t. the space $X$, then the map $\Psi_{\KK}$
(resp. $\bar{\Psi}_{\KK}$) has the HLP w.r.t. the space $X$
for every simplicial complex whose vertex set has cardinality $\alpha$. 
\end{lemma}
\begin{proof} This is a straightforward consequence of the fact that, by definition, the following natural
square diagrams are cartesian, where $S = \bigcup \mathcal{K}$ is the vertex set of $\KK$.
$$
\xymatrix{
L(\Omega,\KK) \ar[d] \ar@{^(->}[r] & L_{\mathrm{f}}(\Omega,S) \ar[d] \\
|\KK|_1 \ar@{^(->}[r] & |\PF^*(S)|_1
} \ \ \ \ \ \ \ \ 
\xymatrix{
\bar{L}(\Omega,\KK) \ar[d] \ar@{^(->}[r] & L(\Omega,S) \ar[d] \\
\overline{|\KK|_1} \ar@{^(->}[r] & \overline{|\PF^*(S)|_1}
}
$$

\end{proof}

Notice that the following lemma applies in particular to every compact metrizable space (e.g. the $n$-spheres). Recall that $\aleph_0$ denotes the cardinality of $\N$.

\begin{lemma} \label{lem:redaleph0} Let $X$ be a separable space. If $\Psi_{\aleph_0}$ (resp. $\bar{\Psi}_{\aleph_0}$) has the HLP w.r.t. the space $X$ then, for every infinite cardinal $\gamma$, the map $\Psi_{\gamma}$
(resp. $\bar{\Psi}_{\gamma}$) has the HLP w.r.t. the space $X$.
\end{lemma}
\begin{proof}
Let $S$ be a set of cardinality $\gamma$,
$H : X \times [0,1] \to |\PF^*(S)|_1$ (resp. $\bar{H} : X \times [0,1] \to \overline{|\PF^*(S)|_1}$) and $h : X \to L_{\mathrm{f}}(\Omega,S)$
(resp. $\bar{h} : X \to L(\Omega,S)$)  be continuous maps such that $\Psi_S \circ h = H(\bullet,0)$ (resp. $\bar{\Psi}_S \circ \bar{h} = \bar{H}(\bullet,0)$). Since $X$ is separable, $X \times [0,1]$ is also separable and so are $H(X \times [0,1])$ and
$\bar{H}(X \times [0,1])$. Let $(x_n)_{n \in \N}$ be a dense sequence of elements of $H(X \times [0,1])$
(resp. $\bar{H}(X \times [0,1])$). Each $\supp(x_n) \subset S$ is countable,
and therefore so is $D = \bigcup_n \supp(x_n)$.

We first claim that, for any
$\alpha \in H(X \times [0,1])$ (resp. $\alpha \in \bar{H}(X \times [0,1])$) we have $\supp(\alpha) \subset D$. Indeed, if $\alpha(s_0) \neq 0$ for some $s_0 \not\in D$, then there exists $x_n$
such that $d(x_n,\alpha) < \alpha(s_0)$. But since $d(x_n,\alpha) = \sum_{s \in S}|\alpha(s)-x_n(s)|$,
this condition implies $x_n(s_0) \neq 0$, contradicting $\supp(x_n) \subset D$. Therefore
$\supp(\alpha) \subset D$ for all $\alpha \in H(X \times [0,1])$ (resp. $\alpha \in \bar{H}(X \times [0,1])$),
and $H$ (resp. $\bar{H}$) factorizes through a map $H_D : X \times [0,1] \to |\PF^*(D)|_1$
(resp. $\bar{H}_D : X \times [0,1] \to \overline{|\PF^*(D)|_1}$)
and the natural inclusion $|\PF^*(D)|_1 \subset |\PF^*(S)|_1$ (resp. $\overline{|\PF^*(D)|_1} \subset \overline{|\PF^*(S)|_1}$).

Notice that this implies that $h$ (resp. $\bar{h}$) takes values
in $L_{\mathrm{f}}(\Omega,D)$ (resp. $L(\Omega,D)$), too. By assumption,
there exists $\tilde{H}_D : X \times [0,1] \to L_{\mathrm{f}}(\Omega,D)$ (resp.
$\tilde{\overline{H}}_D : X \times [0,1] \to L(\Omega,D)$) such that $\Psi_D \circ \tilde{H}_D = H_D$
and with $\tilde{H}_D(\bullet,0) = h$ (respectively, $\tilde{\bar{H}}_D(\bullet,0)=\bar{h}$). Composing $\tilde{H}_D$ (resp. $\tilde{\bar{H}}_D$) with the natural
injection $L_{\mathrm{f}}(\Omega,D) \into L_{\mathrm{f}}(\Omega,S)$ (resp. $L(\Omega,D) \into L(\Omega,S)$) we get the lifting $\tilde{H}$ (resp. $\tilde{\bar{H}}$) we want, and this proves the claim.
$$
\xymatrix{
 & L_{\mathrm{f}}(\Omega,D) \ar[dd]^<<<<<{\Psi_D} | \hole \ar@{^(->}[r] & L_{\mathrm{f}}(\Omega,S) \ar[d]^{\Psi_S} \\
X \times [0,1] \ar@{-->}_{H_D} [dr]\ar@{-->} ^{\tilde{H}_D} [ur] \ar[rr]_{\!\!\!\!\!\!\!\!\!\!\!\!\!\!\!\!\!\!\!\!H} & & |\PF^*(S)|_1 \\
& |\PF^*(D)|_1 \ar@{^(->}[ur] & 
}
$$

\end{proof}

\begin{proposition} \label{prop:HLPaleph0} Let $X$ be a topological space and $\gamma$ a countable cardinal.
Then $\Psi_{\gamma}$ has the HLP property w.r.t. $X$ as soon as $\gamma$ is finite or $X$ is compact.
Moreover $\bar{\Psi}_{\gamma}$ has the HLP w.r.t. $X$ as soon as $X$ is compact.
\end{proposition}
\begin{proof} Let $X$ be an arbitrary topological space. Our cardinal $\gamma$ is the cardinal of some initial segment $S \subset \N = \Z_{\geq 0}$ that is, either $S = [0,m]$ for some $m$,
or $S = \N$. Let  $H : X \times [0,1] \to |\PF^*(S)|_1$ and $h : X \to L_{\mathrm{f}}(\Omega,S)$ 
such that $H(\bullet,0) = \Psi_S  \circ h$. For $(x,u) \in X \times [0,1]$,
the element $H(x,u) \in |\PF^*(S)|_1$ is of the form $(H(x,u)_s)_{s \in S}$,
with $\sum_{s \in S} H(x,u)_s = 1$. 
Since, for each $s \in S$, the map $|\PF^*(S)|_1 \to [0,1]$ given by
$\alpha \mapsto \alpha(s)$ is 1-Lipschitz, the composite
map $(x,u) \mapsto H(x,u)_s$ 
defines a continuous map $X \times [0,1] \to [0,1]$.

Let us choose $x \in X$. We set, with the convention $0/0 = 0$,
$$
a_n(x,u) = \frac{H(x,u)_n}{1- \sum_{k< n} H(x,u)_k} \in [0,1],
\ \ 
A_n(x) = h(x)^{-1}(\{ n \}) \in L(2)
$$
and
 we construct recursively,
for each $n \in \N$,
\begin{itemize}
\item maps $\Omega_n(x,\bullet) : [0,1] \to L(2)$
\item maps $E^{(n)}_{x,\bullet} : [0,1] \to L(2)$
\end{itemize}
by letting 
$$E_{x,u}^{(n)} = \Omega \setminus \bigcup_{k < n} \Omega_k(x,u), \ \  \Omega_n(x,u) = \Phi(a_n(x,\bullet),E^{(n)}_{x,\bullet},A_n(x))(u)
$$
where $\Phi$ is the map afforded by
Proposition \ref{prop:phi}.

In order for this to be defined at any given $n$, one needs to check that $A_n(x) \subset E_{x,0}^{(n)}$ and
$a_n(x,0)\la(E_{x,0}^{(n)}) = \la(A_n(x))$. This is easily checked by induction because, if $\Omega_k, E^{(k)}$ are defined for
$k < n$, then 
$$
\Omega_k(x,0) = \Phi\left(a_n(x,\bullet),E^{(n)}_{x,\bullet},A_n(x)\right)(0) = A_n(x) = h(x)^{-1}(\{ n \})
$$ 
hence 
$$
E_{x,0}^{(n)} = \Omega \setminus \bigcup_{k <n } A_k(x) = h(x)^{-1}(S \setminus [0,n[) \supset h(x)^{-1}(\{ n \}) = A_n(x)
$$
and moreover $\la(A_n(x)) = \la(h(x)^{-1}(\{ n \})) = H(x,0)_n = a(x,0)\la(E_{x,0}^{(n)})$. Therefore these maps
are well-defined.

 From their definitions and the properties of $\Phi$ one gets immediately by induction that
 $$a_n(x,u) \la(E_{x,u}^{(n)}) = H(x,u)_n = \la(\Omega_n(x,u))$$
 for all $(x,u) \in X \times [0,1]$.

 For a given $(x,u)$, the sets $\Omega_n(x,u)$ are essentially disjoint, since $\Omega_n(x,u) \subset E_{x,u}^{(n)} = \Omega \setminus \bigcup_{k<n} \Omega_k(x,u)$, and moreover $\bigcup_n \Omega_n(x,u) = \Omega$ since $\sum_n \la(\Omega_n(x,u)) = \sum_n H(x,u)_n = 1$. Therefore, we can define a map
$\tilde{H} : X \times [0,1] \to L_{\mathrm{f}}(S)$ 
by setting $\tilde{H}(x,u)(t) = n$ if $t \in \Omega_n(x,u)$.
Clearly $(\Psi_S \circ \tilde{H} (x,u))_n =  \la( \Omega_n(x,u) ) = H(x,u)_n$ for all $n$,
hence $\Psi_S \circ \tilde{H}  = H$. Moreover $\tilde{H}(x,0)_n = \Omega_n(x,0) = A_n(x) = h(x)^{-1}(\{ n \}$
hence $\tilde{H}(x,0) = h(x)$ for all $x \in X$. 

Therefore it only remains to prove that $\tilde{H} : X \times [0,1] \to L_{\mathrm{f}}(\Omega,S)$
is continuous.

Let us define the auxiliary maps $\tilde{H}_n : X \times [0,1] \to L(\Omega,\{ 0, \dots, n\})$
by $\tilde{H}_n(x,u)(t) = \tilde{H}(x,u)(t)$ if $\tilde{H}(x,u)(t)< n$,
and $\tilde{H}_n(x,u)(t) = n$ if $\tilde{H}(x,u)(t) \geq n$ -- that is, $\tilde{H}_n(x,u)(t) = \min(n,\tilde{H}(x,u)(t))$.

We first prove that each $\tilde{H}_n$ is continuous. Let $(x_0,u_0), (x,u) \in X \times [0,1]$. We have
$$
d(\tilde{H}_n(x,u),\tilde{H}_n(x_0,u_0)) =\sum_{k=0}^n \int_{ \Omega_k(x_0,u_0)} d((\tilde{H}_n(x,u)(t),\tilde{H}_n(x_0,u_0)(t))\dd t
$$
hence
$$
d(\tilde{H}_n(x,u),\tilde{H}_n(x_0,u_0)) \leqslant \sum_{k=0}^n \ \la\left( \Omega_k(x_0,u_0) \setminus  \Omega_k(x,u)\right)
\leqslant \sum_{k=0}^n \ \la( \Omega_k(x_0,u_0) \Delta  \Omega_k(x,u))
$$
and therefore it remains to prove that the maps $(x,u) \mapsto \Omega_n(x,u)$ are continuous for each $n \in \N$.

We thus want to prove that $\Omega_n(\bullet,\bullet) \in C^0(X \times [0,1], L(2))$, which we identify with the space
$C^0(X,C^0([0,1],L(2))) = C^0(X,C^0(L(2)))$ since $[0,1]$ is (locally) compact. Recall that $\Phi$ is continuous $C^0([0,1]) \times C^0(L(2)) \times L(2) \to C^0(L(2))$.
Moreover, for arbitrary spaces $Y,Z$ and a map $g \in C^0(Y,Z)$, the induced
map $C^0(X,Y) \to C^0(X,Z)$ given by $f \mapsto g \circ f$ is continuous. Letting $Y = C^0([0,1]) \times C^0(L(2)) \times L(2) $
and $Z = C^0(L(2))$, we deduce from $\Phi : Y \to Z$ a continuous map $\hat{\Phi} : C^0(X,Y) \to C^0(X,Z)$, that is
{}
$$
\xymatrix{
\hat{\Phi} : & C^0(X,C^0([0,1]) \times C^0(L(2))) \times L(2))\ar@{=}[d] \ar[r] & C^0(X,C^0(L(2)) \ar@{=}[d] \\
 & C^0(X \times [0,1],[0,1])\times C^0(X \times [0,1],L(2)) \times C^0(X,L(2)) & C^0(X \times [0,1],L(2))}
$$
By induction and because the maps $a_n, A_n$ are clearly continuous for any $n$,
we get that all the maps involved are continuous, through the recursive identities
\begin{itemize}
\item $\Omega_n = \hat{\Phi}(a_n,E_{\bullet,\bullet}^{(n)},A_n(\bullet))$
\item $E_{x,u}^{(n)} = \Omega \setminus \bigcup_{k < n} \Omega_k(x,u)$
\end{itemize}
and this proves the continuity of $\tilde{H}_n$.

If $S$ is finite this proves that $\tilde{H}$ is continuous, because $\tilde{H} = \tilde{H}_n$ for $n$ large enough in this case.
Let us now assume that $S = \N$ and $X$ is compact. We want to prove that the sequence $\tilde{H}_n$
converges uniformly to $\tilde{H}$. Since each $\tilde{H}_n$ is continuous this will prove that $\tilde{H}$ is continuous.
Let $\eps > 0$. Let $U_n = \{ (x,u) \in X \times [0,1] \ | \ \sum_{k \leq n} H(x,u)_k > 1- \eps \}$. Since $H$ is continuous
this defines a collection of open subsets in the compact space $X \times [0,1]$, and since $\sum_{k \leq n} H(x,u) \to 1$
when $n \to \infty$ for any $(x,u) \in X \times [0,1]$, this collection is an open covering of $X \times [0,1]$. By compactness, and because
this collection is a filtration, we have $X \times [0,1] = U_{n_0}$ for some $n_0 \in \N$. But then, for any
$(x,u) \in X \times [0,1]$ and $n \geq n_0$ we have
$$
d(\tilde{H}_n(x,u),\tilde{H}(x,u)) = \la\left( \bigcup_{k > n } \Omega_k(x,u) \right) = \sum_{k > n} H(x,u)_k \leqslant \eps
$$
and this proves the claim.

\end{proof}

\begin{remark} 
We notice that the liftings constructed in the above proof have the following additional property
that, whenever $H(x,\bullet)$ is a constant map for some $x \in X$, then so is the map $\tilde{H}(x,\bullet)$. 
This follows from the fact that the maps $a_r(x,\bullet)$ are constant as soon as $H(x,\bullet)$ is constant, and then one gets by induction on $n$ that
$\Omega_n(x,u) = \Phi(a_n(x,\bullet),E^{(n)}_{x,\bullet},A_n(x))$ is constant in $u$ by the last item of Proposition \ref{prop:phi}, and thus so is $E_{x,u}^{(n)}$.

\end{remark}
Since it is far simpler in this case, we provide an alternative proof for the case of binary random
variables.

\begin{corollary} \label{cor:Hurewics01} The map $\Psi_2 = \Psi_{\{ 0, 1 \}}$ is a Hurewicz fibration.
\end{corollary}

\begin{proof} (alternative proof)
Let $X$ be a space, and $H : X \times [0,1] \to |\PF^*(2)|_1$ and $h : X \to L(\Omega,2)$
such that $H(\bullet,0) = \Psi_2  \circ h$. Note that $|\PF^*(2)|_1 = \{\alpha : \{ 0, 1 \} \to \R_+ \ | \ \alpha(0)+\alpha(1) = 1 \}$ is isometric to $[0,1]$ through the isometry $j : \alpha \mapsto \alpha(1)$, where the metric on $[0,1]$ is the Euclidean one. If $\alpha = \Psi_2 \circ h(x)$, we have $\alpha(0) = 1 - \la(h(x))$, $j(\Psi_2(h(x))) = \alpha(1) = \la(h(x))$.

 Using the map $\mathbf{g}$ of Lemma
\ref{lem:renorm} 
we note that $\la( ^c \mathbf{g}( ^c A,u)) = u  + (1-u) \la(A)= u \la(\Omega) + (1-u) \la(A)$ and
we define, for $A \in L(2)$ and $a \in [0,1]$,
\begin{itemize}
\item $\tilde{\mathbf{g}}(A,a) = 
\mathbf{g}(A,1- a/\la(A))$ if $a < \la(A)$,
\item $\tilde{\mathbf{g}}(A,\la(A)) = A$,
\item 
$\tilde{\mathbf{g}}(A,a) = ^c \mathbf{g}( ^c A,(a- \la(A))/(1- \la(A)))$ if $a > \la(A)$. 
\end{itemize}
We prove that $\tilde{\mathbf{g}} : L(2) \times [0,1] \to L(2)$ is continuous at each $(A_0,a_0) \in L(2)$. The case $a_0 \neq \la(A_0)$ is clear from the continuity of $\mathbf{g}$, as there is an open neighborhood of $(A_0,a_0)$ on which $a - \la(A)$ has constant sign. Thus we can assume $a_0 = \la(A_0)$. Then
$$
d(\tilde{\mathbf{g}}(A,a), \tilde{\mathbf{g}}(A_0,a_0))
= d(\tilde{\mathbf{g}}(A,a), A_0) \leqslant
d(\tilde{\mathbf{g}}(A,a),A) + d(A,A_0)
$$
But, if $ a < \la(A)$ we have by the inequality of Lemma \ref{lem:renorm}
$$
d(\tilde{\mathbf{g}}(A,a),A)= d\left(\mathbf{g}\left(A,1-\frac{a}{\la(A)}\right),\mathbf{g}(A,0)\right) \leqslant\left| 1 - \frac{a}{\la(A)} \right|
$$
and, if $a > \la(A)$, we have, noticing that $A \mapsto ^c A$ is an isometry of $L(2)$ (as $A \Delta B = (^c A)\Delta (^c B)$),
$$
d(\tilde{\mathbf{g}}(A,a),A)= 
d\left( ^c \mathbf{g}\left( ^c A,\frac{a- \la(A)}{1- \la(A)}\right),A\right) =
d\left( \mathbf{g}\left( ^c A,\frac{a- \la(A)}{1- \la(A)}\right),(^c A)\right) \leqslant 
\left|  \frac{a- \la(A)}{1- \la(A)} \right|
$$
which altogether imply
$$
d(\tilde{\mathbf{g}}(A,a), \tilde{\mathbf{g}}(A_0,a_0))
\leqslant d(A,A_0) + \max\left( 
\left| 1 - \frac{a}{\la(A)} \right|,
\left|  \frac{a- \la(A)}{1- \la(A)} \right|\right)
$$
Since the RHS is continuous with value $0$ at $(A_0,a_0)$
with $a_0 = \la(A_0)$, this proves the continuity of $\tilde{\mathbf{g}}$.

It is readily checked that $\la(\tilde{\mathbf{g}}(A,a)) = a$ for all $A,a$.
We then define $\tilde{H} : X \times [0,1] \to L(\Omega,2)$
by $\tilde{H}(x,u) = \tilde{\mathbf{g}}(h(x),j (H(x,u)))$. We have $\la(\tilde{H}(x,u)) = j (H(x,u))$
hence $\Psi_2 \circ \tilde{H} = H$, and $\tilde{H}(x,0) = h(x)$ for all $x \in X$, therefore $\tilde{H}$ provides the lifting we want.
\end{proof}

Altogether, these statements imply the following result, which completes the proof of Theorem \ref{theo:maintheorem}.

\begin{theorem} \label{theo:fibrations}  For an arbitrary simplicial complex $\KK$, the maps $\Psi_{\KK}$ and $\overline{\Psi}_{\KK}$
are Serre fibrations and weak homotopy equivalences. If $\KK$ is finite, then $\Psi_{\KK}$ and $\overline{\Psi}_{\KK}$
are Hurewicz fibrations and homotopy equivalences.
\end{theorem}
\begin{proof} Let $\KK$ be an arbitrary simplicial complex. We first prove that $\Psi_{\KK}$ and $\overline{\Psi}_{\KK}$
are Serre fibrations. By Lemmas \ref{lem:redcardinal} and \ref{lem:redaleph0}, and since the $n$-spheres are separable spaces,
we can restrict ourselves to proving the same statement for $\Psi_{\gamma}$
and $\overline{\Psi}_{\gamma}$ when $\gamma \leq \aleph_0$, and this is true in this case because the $n$-spheres are compact,
by Proposition \ref{prop:HLPaleph0}. Let us now choose  $\{ x _0 \} \in \KK$, and define $\tilde{x}_0 : S \to [0,1]$ to be given by $x_0 \mapsto 1$
and $x \mapsto 0$ if $x \neq x_0$. Then $\tilde{x}_0 \in |\KK|_1$, and the fiber above it  $\Psi_{\KK}^{-1}(\{ \tilde{x}_0 \})=
\overline{\Psi}^{-1}_{\KK}(\{ \tilde{x}_0 \})$ is a point. Since these two maps are Serre fibrations this implies that they
are weak homotopy equivalences.

We now prove the second part of the statement.
If $\KK$ is a finite simplicial complex, by Lemma \ref{lem:redcardinal} and Proposition \ref{prop:HLPaleph0} we get that 
$\Psi_{\KK}$ and $\overline{\Psi}_{\KK}$ are Hurewicz fibrations. Picking again some 
$\{ x _0 \} \in \KK$ we get that the fiber above some point is itself a point, whence the homotopy fiber
of these Hurewicz fibrations is contractible and they are homotopy equivalences. This proves the claim.

\end{proof}


\begin{thebibliography}{00}
\bibitem{DOWKER} C.H. Dowker, {\it Topology of Metric Complexes}, American J. Math. {\bf 74} (1952), 555-577.
\bibitem{KEANE} M. Keane, {\it Contractibility of the automorphism group of a nonatomic measure space}, 
Proc. Amer. Math. Soc. {\bf 26} (1970), 420-422.
\bibitem{CCS} I. Marin, {\it Classifying spaces and measure theory}, preprint 2017, arXiv:1702.01889v1.
\bibitem{CWMILNOR} J. Milnor, {\it On Spaces having the homotopy type of a CW-complex},
Trans. A.M.S. {\bf 90} (1959) 272-280.
\bibitem{SPANIER} E.H. Spanier, {\it Algebraic Topology}, Springer-Verlag, 1966.
\end{thebibliography}
\end{document}